\documentclass[11pt]{article}
\usepackage{amsfonts,amsmath,latexsym,verbatim,amscd,amsthm}

\title{Renormalized area and \\ properly embedded minimal surfaces \\ in hyperbolic
$3$-manifolds}

\author{Spyridon Alexakis \thanks{Email: alexakis@math.princeton.edu}\\
Princeton University \and Rafe Mazzeo \thanks{Email: mazzeo@math.stanford.edu.
Supported by the NSF under Grant DMS-0505709}\\ Stanford University}

\date{}

\newtheorem{theorem}{Theorem}[section]
\newtheorem{proposition}{Proposition}[section]
\newtheorem{corollary}{Corollary}[section]

\newtheorem{remark}{Remark}[section]

\newcommand{\RR}{\mathbb{R}}

\newcommand{\ZZ}{\mathbb{Z}}
\newcommand{\HH}{\mathbb{H}}
\newcommand{\e}{\epsilon}
\newcommand{\del}{\partial}

\newcommand{\calA}{{\mathcal A}}
\newcommand{\calC}{{\mathcal C}}
\newcommand{\calD}{{\mathcal D}}
\newcommand{\calE}{{\mathcal E}}
\newcommand{\calF}{{\mathcal F}}
\newcommand{\calG}{{\mathcal G}}

\newcommand{\calK}{{\mathcal K}}

\newcommand{\calM}{{\mathcal M}}

\newcommand{\calO}{{\mathcal O}}
\newcommand{\calR}{{\mathcal R}}
\newcommand{\calS}{{\mathcal S}}

\newcommand{\calU}{{\mathcal U}}
\newcommand{\calV}{{\mathcal V}}
\newcommand{\calW}{{\mathcal W}}
\newcommand{\frakc}{{\mathfrak c}}

\newcommand{\olg}{\overline{g}}
\newcommand{\olN}{\overline{N}}
\newcommand{\olT}{\overline{T}}
\newcommand{\olnu}{\overline{\nu}}
\newcommand{\olk}{\overline{k}}
\newcommand{\Ric}{\mathrm{Ric}}
\newcommand{\tr}{\mathrm{tr}\,}

\newcommand{\wh}{\widehat}

\newcommand{\whk}{\wh{k}}

\begin{document}

\maketitle

\begin{abstract} We study the renormalized area functional $\calA$ in the AdS/CFT correspondence,
specifically for properly embedded minimal surfaces in convex cocompact hyperbolic $3$-manifolds 
(or somewhat more broadly, Poincar\'e-Einstein spaces).
Our main results include an explicit formula for the renormalized area of such a minimal surface
$Y$ as an integral of local geometric quantities, as well as formul\ae\ for the first and 
second variations of $\calA$ which are given by integrals of {\it global} quantities over
the asymptotic boundary loop $\gamma$ of $Y$. All of these formul\ae\ are also obtained for
a broader class of nonminimal surfaces. The proper setting for the study of this functional
(when the ambient space is hyperbolic) requires an understanding of the moduli space of all 
properly embedded minimal surfaces with smoothly embedded asymptotic boundary. We show
that this moduli space is a smooth Banach manifold and develop a $\ZZ$-valued degree theory
for the natural map taking a minimal surface to its boundary curve. We characterize the 
nondegenerate critical points of $\calA$ for minimal surfaces in $\HH^3$, and finally, discuss
the relationship of $\calA$ to the Willmore functional. 
\end{abstract}

\section{Introduction}
There is an interesting nonlinear asymptotic boundary problem in which one seeks a minimal submanifold
in hyperbolic space with prescribed asymptotic boundary a submanifold in the sphere at infinity.
This was treated conclusively by Anderson \cite{An1}, \cite{An2} in the early 1980's using techniques from 
geometric measure theory; the solutions he obtains are absolutely volume minimizing with respect to compact 
variations. One may also pose this problem when the ambient space is a convex cocompact hyperbolic manifold, 
or even more generally a conformally compact manifold $(M^{n+1},g)$ (all definitions are reviewed in \S 2), 
and it is not hard to extend the existence theory to these settings. Here, however, we focus mostly on the special 
case of properly embedded minimal surfaces $Y^2$ in $M = \HH^3/\Gamma$ where $\Gamma$ is a convex cocompact 
subgroup (a particular case is $M = \HH^3$ itself), with boundary curve $\del Y = \gamma$ an embedded closed 
curve $\gamma \subset \del M$. Beyond Anderson's aforementioned work, in this particular setting there is also a 
rich existence theory of minimal (not necessarily minimizing) surfaces of arbitrary genus by de Oliveira and
Soret \cite{OS}, see also Coskunzer \cite{Cos1}, \cite{Cos2}. 

It turns out that there is a well-defined Hadamard regularization of the area of such a minimal surface,
and this renormalized area is our central concern. Roughly speaking, our goal is to obtain a local
formula for this renormalized area, i.e.\ one involving integrals of local geometric quantities, and
then to use this to study the variational theory of the renormalized area functional $\calA$. In order to do
this properly, we must study the moduli space of all properly embedded minimal surfaces with embedded
asymptotic boundary, as this is the natural domain of $\calA$. This leads to a subsidiary investigation
of the structure of these moduli spaces and the degree theory for the natural map taking a minimal surface 
to its boundary curve. We also consider the renormalized area of a larger class of nonminimal surfaces. 
We calculate the first and second variations of $\calA$; interestingly, these are expressed as integrals of 
global quantities over the boundary curve. While we do not touch on all aspects of this variational
theory, we are able to characterize the nondegenerate critical points amongst surfaces in hyperbolic
space, and give some estimates for the numerical range of $\calA$. Finally, we show the relationship of
$\calA$ to the much-studied Willmore functional $\calW$, which suggests that $\calA$ is the correct conformally
natural generalization of $\calW$ to surfaces with boundary. 

There are strong motivations from the AdS/CFT correspondence in string theory for studying the renormalized area, 
and we shall explain some of these below, after describing the mathematical context more carefully.

Our results about minimal surfaces parallel a number of known results about Poincar\'e-Einstein (PE) spaces,
so we describe these together. Let $(M,g)$ be a PE space; this means that $M$ is a manifold with boundary,
$g = \rho^{-2}\olg$ where $\rho$ is a boundary defining function for $\del M$ and $\olg$ is smooth and nondegenerate
up to the boundary, and $g$ is Einstein. There is a well-defined conformal class $\frakc(g)$ on $\del M$, called 
the conformal infinity of $g$, which should be regarded as the asymptotic boundary value of $g$.
The space of all PE metrics (with some fixed regularity) on the interior of a given manifold with boundary $M$ is a
Banach manifold, and the conformal infinity map from this to the space of conformal structures on $\del M$ (which
also has the structure of a Banach manifold) is Fredholm of degree $0$. These facts were proved by Anderson \cite{An4},
see also Biquard \cite{Bi} and Lee \cite{Lee}. Most existence results for PE metrics are perturbative in nature,
but Anderson established a scheme to obtain a much broader existence theory when $\dim M = 4$ using degree theory \cite{An5}.
One key ingredient is the properness of this conformal infinity boundary value map over the preimage of scalar positive 
conformal classes on $\del M$. There are substantial technicalities in making all of this work; recent work of Chang 
and Yang \cite{CY} clarifies some of this.

We first prove an analogous result for properly embedded minimal submanifolds:
\begin{theorem}
Suppose that $M = \HH^3/\Gamma$ is a convex cocompact hyperbolic manifold, and let $\calM_k(M)$ be the space of
properly embedded minimal surfaces in $M$ of genus $k$ with asymptotic boundary curve a $\calC^{3,\alpha}$ embedded
closed (but possibly disconnected) curve in $\del M$. Let $\calE$ denote the space of all $\calC^{3,\alpha}$ closed
embedded curves in $\del M$. Then both $\calM_k(M)$ and $\calE$ are Banach manifolds, and the natural map
\[
\Pi: \calM_k(M) \longrightarrow \calE
\]
is a smooth proper Fredholm map of index $0$. 
\end{theorem}
These properties of $\Pi$ imply the existence of a $\ZZ$-valued degree for it, which yields many refinements of 
the existence theory for these minimal surfaces. Some consequences will be described in \S 4.  

The proof of most of this uses various well-known tools, hence this can be regarded as a good toy model for the 
corresponding result about four-dimensional Poincar\'e-Einstein spaces. Note that the use of degree theory for the 
boundary map of minimal surfaces goes back to work of Tromba \cite{Tr} in the 1970's and White \cite{Wh2} in the 
1980's, and indeed those papers provided some of the inspiration for Anderson's proposal to use degree theory in 
the Einstein setting. A special case of this degree theory, for genus zero surfaces, was developed in \cite{Cos1}. 

Now return to the PE setting. Assuming the conformal infinity of the PE metric $g$ is sufficiently regular, then
$g$ itself has an expansion up to some order at the boundary. When $\dim M = 4$, this has the form
\begin{equation}
g = \frac{dx^2 + h(x)}{x^2}, \qquad h(x) \sim h_0 + x^2 h_2 + x^3 h_3 + \ldots;
\label{eq:fefgrexp}
\end{equation}
here each $h_j$ is a symmetric $2$-tensor on $\del M$; in particular, $h_0$ is a metric representing $\frakc(g)$
and $h_3$ is trace- and divergence-free with respect to $h_0$. All other $h_j$ are determined in terms of
these two tensors. Furthermore, $x$ is a special boundary defining function naturally associated to the choice of $h_0$.
The volume form $dV_g$ has a corresponding expansion
\[
dV_g \sim \frac{A_0}{x^4} + \frac{A_2}{x^2} + A_4 + \ldots;
\]
the $x^{-3}$ and $x^{-1}$ terms are absent due to the absence of the $h_1$ term and the vanishing trace
of $h_3$. The volume of $\{x \geq \e\}$ is obviously finite for each $\e > 0$ and has an expansion
as $\e \searrow 0$ of the form
\[
\mbox{Vol}\,(\{x \geq \e\}) \sim \frac{\alpha_0}{\e^3} + \frac{\alpha_1}{\e} + \calV(M,g) + \ldots.
\]
The constant term in this expansion is by definition the renormalized volume of $(M,g)$. The key fact, 
first proved by the physicists Henningson and Skenderis \cite{HS}, cf.\ \cite{Gr} for a careful mathematical 
treatment, is that this is well-defined independently of the choice of metric $h_0 \in \frakc(g)$. The
definition of renormalized volume extends to arbitrary dimensions, and they show that it is well defined 
when $\dim M$ is even; when $\dim M$ is odd, however, it is not well-defined and has a simple transformation 
law under change of representative $h_0$. For simplicity here we focus on the four-dimensional case.  Using the
Einstein condition in the Gauss-Bonnet formula, Anderson \cite{An4} noted that
\begin{equation}
\calV(M,g) = \frac{4\pi^2}{3} \chi(M) - \frac{1}{6} \int_M |W|^2\, dV_g;
\label{eq:GB4d}
\end{equation}
here $W$ is the Weyl tensor, and the integral is convergent since $|W|^2$ is pointwise conformally invariant of 
weight $-4$. Anderson also computed a formula giving the infinitesimal variation of the renormalized volume in 
the direction of an infinitesimal Einstein deformation $\kappa$:
\begin{equation}
\left. D\calV \right|_g(\kappa) = -\frac{1}{4} \int_{\del M} \langle \kappa_0, h_n \rangle\, dV_{h_0},
\label{eq:varrevol}
\end{equation}
in terms of the leading term in the expansion $\kappa \sim \kappa_0 + x \kappa_1 + \ldots$. A much easier derivation
of this formula is given in \cite{Alb}.  Again, this extends immediately to all even dimensions.
It follows from this that (when $n$ is even), $\HH^{n}$, and indeed any convex cocompact hyperbolic quotient
$\HH^{n}/\Gamma$, is a critical point of $\calV$. The variational problem for renormalized volume remains unstudied.
When $\dim M = 4$, $\calV$ is closely related to the $\sigma_2$ functional of the underlying incomplete metric
on $M$, and there are some interesting rigidity results using it, see \cite{CGY}. There are also several nicely 
geometric results about renormalized volume in $3$ dimensions \cite{W}, \cite{KS} (recall that it depends on some 
choices here, so one has not simply a number but rather a functional on a given conformal class of the boundary surface).  

Shortly after \cite{HS}, and motivated by the same string-theoretic concerns, Graham and Witten \cite{GW} proved the 
existence of a well-defined renormalized area $\calA$ for properly embedded minimal submanifolds $Y$ in a PE space where
the boundary of $Y$ is also embedded in $\del M$. Two dimensions is critical for minimal surfaces in roughly the same 
way that four dimensions is critical for Einstein metrics, so it is reasonable that the results above about renormalized 
volume of four-dimensional PE metrics have analogues for properly embedded minimal surfaces, and this is indeed true.
Our second main result is an explicit formula for $\calA$ and its first and second variations:
\begin{theorem}
Let $Y \in \calM(M)$ have a $\calC^{3,\alpha}$ embedded boundary curve $\gamma$. Then
\[
\calA(Y) = -2\pi \chi(Y) - \frac12 \int_Y |\whk|^2\, dA,
\]
where $\whk$ is the trace-free second fundamental form of $Y$; the integral is convergent since $|\whk|^2\, dA$ is
invariant under conformal changes of the ambient metric. Furthermore, if $0$ is not in the spectrum of the Jacobi
operator $L_Y$ (in which case we say that $Y$ is non-degenerate), and $\dot{\phi}$ is a Jacobi field on $Y$ (i.e.\ $L_Y 
\dot{\phi} =0$), which thus corresponds to a one-parameter family of minimal surfaces around $Y$, then (relative to
a normalization which will be explained later), $\dot{\phi} \sim \dot{\phi}_0 + x \dot{\phi}_1 + \ldots$ and
\[
\left. D\calA\right|_Y(\phi) = -3 \int_{\gamma} (\dot{\phi})_0 \, u_3\, ds;
\]
here $u_3$ is the coefficient of $x^3$ in the expansion for the function $u$ which gives a graph parametrization of $Y$
over the vertical cylinder $\gamma \times [0,x_0)$. Furthermore,
\[
\left. D^2\calA\right|_Y(\phi,\phi) = - \frac12 \int_\gamma \dot{\phi}_0\, \dot{\phi}_3\, ds;
\]
as we explain later, this shows that the Hessian of $\calA$ is represented by the Dirichlet-to-Neumann
operator for the Jacobi operator $L_Y$. 
Finally, if $M = \HH^3$, the unique nondegenerate critical points of $\calA$ are the totally geodesic
copies of $\HH^2$ (so $\gamma$ is a round circle).
\end{theorem}

We now turn to the physical precursors of all of this. Maldacena's pioneering work \cite{m:wllNt} proposes that 
in the large t'Hooft coupling regime, the expectation value of the Wilson loop operator corresponding to some closed 
loop $\gamma \subset \del M$ should be given by the area of the minimal surface $Y\subset M^{n+1}$ with asymptotic
boundary $\gamma$. The papers \cite{m:wllNt} and \cite{dgo:wlms} already point out that one must introduce an 
area renormalization, which motivated \cite{GW}. 

Quite recently, it has also been suggested (\cite{rt:hdeefa}, \cite{rt:ahee}) that this renormalized area 
be used to measure the entanglement entropy of a particular region in the CFT. More specifically, \cite{rt:hdeefa} 
(see also \cite{ht:assee}) proposes an `area law': for the model $(\HH^{n+1}, \mathbb{S}^n)$, the information of 
a domain $\Omega\subset \mathbb{S}^n$ should correspond to the region in $\HH^{n+1}$ enclosed by a minimal 
submanifold with asymptotic boundary $\partial \Omega$ (which need not be well-defined, of course, since the minimal
submanifold is not unique), and in particular, the entanglement entropy of a domain $\Omega \subset \mathbb{S}^2$ 
should correspond to the renormalized area of the minimal surface $Y\subset \HH^3$ with boundary $\gamma=\partial 
\Omega$ (see formula (1.5) in \cite{rt:hdeefa}). This assertion is checked in the lowest dimensional case $n=1$ 
in \cite{rt:ahee}, \cite{rt:hdeefa}, and special examples are also presented in \cite{ht:assee} for $n=2$ -- but 
the validity of the assertion in higher dimensions is disputed in \cite{st:eetah}. 

Motivated by these proposals, substantial effort has been devoted in several recent physics papers to understanding
the geometric features and renormalized area of various simple cases of minimal surfaces in $\HH^3$. For example,  
in \cite{dgo:wlms} the authors compute the renormalized area of totally geodesic planes; in \cite{ht:assee} 
Hirata and Takayanagi study the existence of minimal surfaces with two disconected circles as asymptotic boundaries 
and also estimate the renormalized area of those surfaces; Maldacena \cite{m:wllNt} studies the case of a rectangle 
where the length $T$ of one side approaches infinity. Furthermore, Drukker-Gross-Ooguri, \cite{dgo:wlms} and 
Polyakov-Rychkov, \cite{pr:ldacc} have sought to check the proposed formula in \cite{m:wllNt} relating
the expectation value of the Wilson loop in CFT with the renormalized area of a minimal surface in 
AdS. This verification involves calculating the first and second 
variations of the renormalized area functional with respect to deformations of the loop $\gamma$. 
Since those authors did not have a usable explicit formula for the renormalized area, their calculations 
required justification for dropping certain divergent terms; in contrast, our local formul\ae\  allow
for straightforward calculations.

Since there seems to be active and continuing interest in these proposals relating renormalized area 
with the expectation values of Wilson loop operators, the loop equation and to entanglement entropy,
we hope that our results will facilitate further investigations in this area. 

Our paper is structured as follows: in section \S 2 we present some  background material needed for this work, 
on Poincar\'e-Einstein metrics, uniformly degenerate elliptic operators and embedded minimal surfaces in convex
co-compact hyperbolic 3-manifolds. The local formula for $\calA$ is proved in section \S 3, and
certain global aspects of this functional are studied in \S 5. The intervening \S 4 develops the moduli space
theory of properly embedded minimal surfaces. This provides the correct setting in which to derive the first 
and second variation formul\ae, which appears in \S 6. \S 7 characterizes the nondegenerate critical points 
of $\calA$ when the ambient space is hyperbolic $3$-space, $\HH^3$. Finally, in \S 8 we discuss the relationship 
of $\calA$ and the Willmore functional.

The first author is very grateful to Chris Herzog, Juan Maldacena and A. M. Polyakov for useful conversations. 
The second author wishes to thank Joel Hass, Steve Kerckhoff and particularly Brian White for helpful conversations.

\section{Geometric and analytic preliminaries}
We now give precise definitions of the spaces and submanifolds we shall be working with and explain
some of their properties. We also discuss some basic results about elliptic operators on these spaces.

\subsection{Conformally compact and Poincar\'e-Einstein spaces and convex cocompact hyperbolic $3$-manifolds}
A Riemannian manifold $(M,g)$ is called conformally compact if $M$ is the interior of a smooth compact manifold
with boundary and $g = \rho^{-2} \olg$ where $\rho$ is a defining function for $\del M$ and $\olg$ is a metric
smooth and nondegenerate up to $\del M$. Any such metric is complete and has sectional curvatures tending to
$-|d\rho|_{\olg}^2(q)$ upon approach to any point $q \in \del M$. In particular, if $|d\rho|^2_{\olg}$ is
constant along $\del M$, we say that $(M,g)$ is asymptotically hyperbolic (AH).  To any conformally compact
metric $g$ one may associate a conformal class on $\del M$:
\[
\frakc(g) = \left[ \left. \rho^2 g \right|_{T\del M}
\right],
\]
which is obviously independent of the choice of defining function $\rho$. This conformal equivalence class
is called the conformal infinity of $g$.

Any AH metric has a normal form, due to Graham and Lee \cite{GL}. Let $(M,g)$ be an AH space and fix any
metric $h_0$ representing the conformal class $\frakc(g)$. Then there is a unique defining function $x$
for $\del M$, defined in some neighborhood $\calU$ of the boundary, which satisfies the two conditions
\[
|d\log x|_g^2 \equiv 1, \qquad \left. \olg \right|_{T\del M} = h_0, \ \ \mbox{where}\ \ \olg = x^2 g.
\]
The flow lines for the gradient $\nabla^{\olg}x$ give a product decomposition $\calU \cong [0,x_0) \times \del M$,
in terms of which the pullback of the metric $g$ takes the form
\begin{equation}
g = \frac{dx^2 + h(x)}{x^2}, \qquad h(x) \sim h_0 + x h_1 + x^2 h_2 + \ldots
\label{eq:GLnf}
\end{equation}
The defining function $x$ associated to the boundary metric $h_0$ will be called a special boundary defining
function (bdf).

A case of particular special interest is when $(M,g)$ is Poincar\'e-Einstein (PE), which means simply that it is
both conformally compact and Einstein. These metrics were introduced by Fefferman and Graham \cite{FG} as a way of
canonically associating a Riemannian metric on an ambient $(n+1)$-manifold to a conformal class on an $n$-manifold,
with the goal of finding new conformal invariants on the boundary via Riemannian invariants of the ambient manifold.
If the conformal infinity of such a $g$ is smooth, then the family of tensors $h(x)$ in (\ref{eq:GLnf})
has a complete expansion
in powers of $x$ (and also powers of $x^{n-1}\log x$ when $n = \dim X$ is odd, $n \geq 5$). The coefficients
$h_0$ and $h_{n-1}$ are formally undetermined, but all other $h_j$ can be expressed as local differential operators
applied to these two coefficients; it is thus natural to think of the pair $(h_0,h_{n-1})$ as the Cauchy data of $g$.

In this paper we shall be primarily concerned with the
three-dimensional case. If $(M^3,g)$ is PE, then $M$ is isometric
to a convex cocompact quotient $\HH^3/\Gamma$. (Convex cocompact
means that $\Gamma$ is geometrically finite and has no parabolic
elements; equivalently, the quotient by $\Gamma$ of the convex
hull (in $\HH^3$) of the limit set $\Lambda(\Gamma)$ is compact in
$M$.) The Fefferman-Graham expansion for $g$ simplifies then, and
has a special form where only $h_0$, $h_2$ and $h_4$ are nonzero,
see \cite{Ep} and Epstein's appendix in \cite{PP}. These
coefficients can be calculated in terms of the metric and second
fundamental form of any one of the level sets $\{x =
\mbox{const.}\,\}$, and the special bdf $x$ then has the property
that $-\log x$ is the distance function to this level set (up to
an additive constant).

\subsection{Uniformly degenerate operators}
We shall be using results about the mapping and regularity properties for elliptic operators which
are uniformly degenerate. The theory here is drawn from \cite{Ma-edge}, but see also \cite{Lee}.

Let $X$ be a manifold with boundary, and suppose that $(x,y)$ is a local chart near some boundary point,
where $x$ is a boundary defining function and $y$ restricts to coordinates along the boundary. A differential
operator $L$ is called uniformly degenerate if in any such chart it takes the form
\[
L = \sum_{j+|\alpha|\leq m} a_{j,\alpha}(x,y) (x\del_x)^j (x\del_y)^\alpha.
\]
We assume that the coefficients are smooth, or at least $\calC^{2,\alpha}$ up to $\del X$. There is a well-defined
uniformly degenerate symbol
\[
{}^0\sigma_m(L)(x,y,\xi,\eta) = \sum_{j+|\alpha| =m} a_{j,\alpha}(x,y)\xi^j \eta^\alpha
\]
and $L$ is elliptic in this category of objects if this symbol is
invertible for all $(x,y)$ and $(\xi,\eta) \neq 0$. Unlike in the
standard interior case, there is a further model which must be
studied, called the normal operator, which is defined by
\[
N(L) = \sum_{j+|\alpha| \leq m} a_{j,\alpha}(0,y) (t\del_t)^j (t\del_v)^\alpha,
\]
where $(t,v)$ are linear coordinates on the half-space $\RR^+_t \times \RR^\ell_v$, $\ell + 1 = \dim X$.
Finally, for any such operator, we define its set of indicial roots to be the values of $\mu$
for which $L x^\mu = \calO(x^{\mu + 1})$. (This definition must be modified slightly when $L$ is
a system.)  These values are the roots of the indicial polynomial $\sum_{j \leq m} a_{j0}(0,y) \mu^j$,
so (in the scalar case) there are exactly $m$ such values. For simplicity, we now restrict to the case
where the degree of $L$ is $2$, and list the indicial roots as $\mu_1$ and $\mu_2$.

We shall let these operators act on weighted Sobolev and H\"older spaces of functions. By definition
$H^k_0(X)$ consists of functions which lie $L^2$ along with all derivatives up to order $k$ with
respect to the vector fields $x\del_x$ and $x\del_y$. Similarly, $\Lambda^{k,\alpha}_0$ denotes the
H\"older space where the derivatives and difference quotients are measured with respect to these same
vector fields. If $E$ is any function space, then $x^\mu E$ denotes the set of functions
$x^\mu v$ where $v \in E$.

The basic result we need is the following:
\begin{proposition}
Let $L$ be a uniformly degenerate operator of degree $2$ on the compact manifold with boundary $X$,
and suppose that $L$ is uniformly degenerate elliptic. If $N(L): t^{\mu - 1/2}H^2_0(dtdv) \longrightarrow
t^{\mu-1/2}L^2(dtdv)$ is an isomorphism for one value of $\mu \in (\mu_1,\mu_2)$, then it is an
isomorphism for every $\mu \in (\mu_1,\mu_2)$, and for all such $\mu$,
\[
L: x^\mu \Lambda^{k+2,\alpha}_0(X) \longrightarrow x^{\mu} \Lambda^{k,\alpha}_0(X)
\]
is Fredholm, with nullspace contained in $x^{\mu_2}\Lambda^{\ell,\alpha}_0(X)$ for every $\ell$.
If $N(L)$ is only surjective as a map $t^{\mu-1/2}H^2_0(dtdv) \to t^{\mu-1/2}L^2(dtdv)$ but its
nullspace is nontrivial, then $L$
itself still has closed range of finite codimension, but an infinite dimensional kernel.
\label{pr:eeo}
\end{proposition}
The proof is contained in \cite{Ma-edge}.

\subsection{Properly embedded minimal surfaces with embedded asymptotic boundary}
\label{vms}
As explained in the introduction, there is a rich existence theory for properly embedded minimal
or area-minimizing surfaces in convex cocompact hyperbolic $3$-manifolds. Something not treated
in Anderson's original investigations is the boundary regularity. One expects that a properly embedded
minimal surface $Y$ is as regular as its asymptotic boundary curve $\gamma$.  This problem and
its generalization to higher dimensional minimal codimension one submanifolds was investigated by
Lin \cite{Lin}, Hardt and Lin \cite{HL} and Tonegawa \cite{To}. The higher codimension case has apparently
not been treated at all, but is in fact not so difficult using the theory of uniformly degenerate
elliptic operators; we shall come back to this in a later paper. In general dimensions and codimensions,
if $\gamma$ is smooth then any corresponding minimal $Y$ with $\del Y = \gamma$ is polyhomogeneous
at the  boundary, i.e.\ has an expansion in powers of any defining function for $M$ restricted to $Y$;
when $\dim Y$ is even, only positive integer powers appear, while if $\dim Y$ is odd, then powers of
$x^k \log x$ also appear; all of this is completely analogous to the situation for PE metrics. The case
of importance here, however, is covered by the various papers cited above:

\begin{proposition}
\label{tonegreg}
Let $\gamma$ be $\calC^{k,\alpha}$ embedded curve in $\del M$, where $M = \HH^3/\Gamma$ is convex
cocompact, $k \in {\mathbb N}$, $0 < \alpha < 1$. Then $Y$ is $\calC^\infty$ in the
interior of $M$ and $\calC^{k,\alpha}$ up to $\gamma = \overline{Y} \cap \del M$.
\end{proposition}

We discuss some features of the proof in order to bring out some consequences.
This result is local in $\gamma$, so we may as
well suppose that $M = \HH^3$ and focus on the behaviour of $Y$ near some fixed point $p \in \gamma$. Using
the upper half-space model with coordinates $y \in \RR^2$, $x > 0$, place $p$ at the origin and choose a
local arc-length parametrization $\gamma(s)$ for $\gamma$ (with respect to the standard Euclidean metric
on $\RR^2$). Let $\Gamma$ denote the vertical cylinder over $\gamma$, i.e.\ $\Gamma = \{(y,x) \in \RR^2
\times \RR^+: y \in \gamma\}$; thus near the origin, $\Gamma = \{(\gamma(s),x)\}$.

Choose two smooth families of minimal hemispheres, i.e.\ totally geodesic copies of $\HH^2$, which
lie completely inside and outside of $\gamma$, respectively, and which are tangent to $\gamma$,
and let $\Gamma^{\pm}$ be the envelopes of these families. These are smooth mean-convex surfaces
tangent to $\Gamma$ along $\gamma$, and it is straightforward to use them as barriers to deduce
that $Y$ must lie in the open set between $\Gamma^-$ and $\Gamma^+$. It follows that $Y$ is
vertical along $\gamma$, or equivalently, that its unit normal with respect to the Euclidean metric
on the upper half-space is tangent to $\RR^2 = \{x=0\}$ along $\gamma$. We now write $Y$ as
a horizontal graph over $\Gamma$. More specifically, if $\olN = \olN(s)$ is the unit
normal (again with respect to the Euclidean metric) at a point of $\Gamma$, then there is
a scalar function $u(s,x)$ and a neighbourhood $\calU$ of the origin so that
\[
Y \cap {\mathcal U} =  \{ F(s,x):= (\gamma(s) + u(s,x)\olN(s),x) : |s| < \e, x < \e\}.
\]
The argument above implies that $u(s,0) = \del_s u(s,0) = 0$.

The regularity of $Y$ along $\gamma$ is equivalent to that of this function $u$, and the key point is that $u$
is a solution of a uniformly degenerate elliptic partial differential equation ${\mathcal F}(u) = 0$
corresponding to the minimality of $Y$, which we derive now. The function $F$ induces a coordinate chart on $Y$;
let the indices $1$ and $2$ refer to the $s$ and $x$ coordinates, respectively. Letting $\olT = \gamma'(s)$, then
\[
F_s = (1-\kappa u)\olT + u_s \olN, \qquad F_x = u_x \olN + \del_x,
\]
where $\kappa$ is the curvature of $\gamma$. For convenience below, write $w = 1 - \kappa u$. The
inward pointing $\olg$ unit normal is equal to
\begin{equation}
\olnu = \frac{F_x \times F_s}{|F_x\times F_s|} = J^{-1}(-u_s \olT + w \olN - u_x w \del_x),
\qquad J := \sqrt{u_s^2 + w^2(1+u_x^2)}.
\label{eq:defnu}
\end{equation}
The coefficients of the first fundamental form and its inverse are
\[
(\olg_{ij}) =
\begin{bmatrix}
  w^2 + u_s^2 & u_x u_s \\ u_x u_s & 1+u_x^2
\end{bmatrix},
\qquad \mbox{and}\qquad (\olg^{ij}) = \frac{1}{J^2}
\begin{bmatrix}
  1+u_x^2 & -u_x u_s \\ -u_x u_s & w^2 + u_s^2
\end{bmatrix}.
\]
Next, we compute that
\[
F_{ss} = (w_s - \kappa u_s) \olT + (u_{ss} + \kappa w) \olN,\quad
F_{sx} = -\kappa u_x \olT + u_{xs}\olN, \quad F_{xx} = u_{xx} \olN,
\]

\[
(\olk_{ij}) = -\frac{1}{J}
\begin{bmatrix}
 w(u_{ss} + \kappa w) - u_s (w_s - \kappa u_s) & w u_{xs} + \kappa u_x u_s \\
w u_{xs} + \kappa u_x u_s & w u_{xx}
\end{bmatrix}.
\]
Finally, use the general formula $k_{ij} = e^{\phi}(\olk_{ij} +\del_{\olnu} \phi \, \olg_{ij})$ relating the second
fundamental forms of $Y$ of two conformally related metrics $g = e^{2\phi}\olg$. Here $\phi = -\log x$ and
$\olnu$ is as in (\ref{eq:defnu}), so the matrix $(k_{ij})$ is equal to
\[
-\frac{1}{Jx}
\begin{bmatrix}
  w(u_{ss} + \kappa w) - u_s (w_s - \kappa u_s) - x^{-1}u_x w(w^2 + u_s^2) &
w u_{xs} + \kappa u_x u_s - x^{-1}u_x^2 u_s w \\
w u_{xs} + \kappa u_x u_s - x^{-1}u_x^2 u_s w  & w u_{xx} - x^{-1} u_x w (1+u_x^2)
\end{bmatrix}.
\]
The equation of minimality, i.e.\ that $g^{ij}k_{ij} = H = 0$, is then given by the expression
\begin{equation}
\begin{array}{rcl}
\calF(u) & := &  (1+u_x^2)\left[ w(u_{ss} + \kappa w) - u_s (w_s - \kappa u_s)\right]
- 2 u_x u_s\left( w u_{xs} + \kappa u_x u_s\right) \\
& + & w (w^2 + u_s^2)u_{xx} -2 \frac{w u_x}{x}\left( u_s^2 + w^2 (1 + u_x^2)\right) = 0.
\end{array}
\label{eq:mse}
\end{equation}
The coefficient $1/x$ in this last term makes this a degenerate elliptic equation.

Assume that $\gamma$ is at least $\calC^3$; we compute the first few coefficients in the expansion
of $u(s,x)$ as $x \searrow 0$. Set $u \sim u_2(s) x^2 + u_3(s) x^3 + \ldots$ (since we already know that
$u$ vanishes to second order). Inserting this into $\calF(u) = 0$ yields that $u_2(s) = \frac12 \kappa(s)$,
but $u_3(s)$ is formally undetermined by the equation. In other words, this coefficient must depend
globally on $Y$. Just as in the Fefferman-Graham expansion for PE metrics, all higher terms in the expansion
for $u$ are determined by $\gamma$ and $u_3$ and their derivatives, so we regard $(\gamma, u_3)$ as the Cauchy
data for the minimal surface $Y$. Using the unique continuation theorem from \cite{Ma-uc}, it is straightforward
to show that if $Y_1$ and $Y_2$ are two minimal surfaces with the same Cauchy
 data $(\gamma,u_3)$ (even locally), then $Y_1 \equiv Y_2$.
This global coefficient $u_3$ plays a central role in our work.

As a side remark for the moment, consider $\calC^3$ surfaces with boundary $\overline{Y} \subset \overline{M}$ with
$\del Y \subset \del M$, which intersect $\del M$ orthogonally (this makes sense since $\overline{M}$ has a conformal
structure). Any such $Y$ can still be represented near the boundary as a normal graph over the
vertical cylinder $\Gamma$ over its boundary curve $\gamma$, and the graph function
still vanishes to second order. It is no longer necessarily true that $u_2 = \frac12 \kappa$.
The second fundamental form now satisfies
\[
(k_{ij}) = \begin{bmatrix}
\frac{1}{x}(2u_2-\kappa) + 3u_3+\calO(x) & -2u_2' + \calO(x^2) \\ -2u_2' + \calO(x^2) & -3u_3 + \calO(x)
\end{bmatrix};
\]
note that we now have only $|k|_g = \calO(x)$ unless $2u_2 =
\kappa$ in which case $|k|_g = \calO(x^2)$.

Finally, since the Jacobian term $J = |F_s \times F_x| = 1 + \calO(x^2)$, we see that in these
coordinates, the area form equals
\[
dA = \frac{1 + \calO(x^2)}{x^2}\, dsdx.
\]
Writing $Y_\e = Y \cap \{x \geq \e\}$, then by definition, the renormalized area of $Y$ is
the constant term in the expansion
\begin{equation}
\int_{Y_\e} dA = \frac{\mbox{length}(\gamma_\e)}{\e} + \calA(Y) +
\calO(\e). \label{eq:defrenarea}
\end{equation}
In order for this to be interesting, we must show that $\calA(Y)$ is well-defined, independently of the
choice of special bdf $x$. This was done by Graham and Witten \cite{GW}; their key observation, which
is particularly simple in this low-dimensional setting, is that if $h_0$ and $\widehat{h}_0 = e^{2\chi_0}h_0$ are
two representatives of the conformal class $\frakc(g)$, corresponding to special bdf's $x$ and $\widehat{x}$,
respectively, then $\widehat{x} = e^{\chi}x$, where $\chi(x,y) = \calO(x^2)$. This means that $\widehat{x} = x + \calO(x^2)$,
and hence in the new coordinate system $(\widehat{s},\widehat{x})$ on $Y$, one still has $dA =
\widehat{x}^{\,-2}(1 + \calO(\widehat{x}^2))\, d\widehat{s}d\widehat{x}$.  From this, the claim about
well-definedness of $\calA(Y)$ is immediate.

The only property about $Y$ needed for this argument to work is
that it is at least $\calC^2$ and meets $\del M$ orthogonally.
Thus even in this broader setting there is still a well-defined
notion of renormalized area of $Y$. To maintain the distinction,
we shall denote this extended renormalized area functional by
$\calR$ rather than $\calA$ when the surface $Y$ is not minimal.

\section{A formula for renormalized area}
We now express the renormalized area of a properly embedded minimal surface $Y$ in $M$ in terms of its Euler
characteristic and an integral of local invariants. In fact, since it is not much more complicated to do so,
we find an expression for the renormalized area when $Y$ lies in an arbitrary Poincar\'e-Einstein space of any
dimension and is not necessarily minimal, but still meets $\del M$ orthogonally.

\begin{proposition}
\label{malaki} Let $(M^{n+1},g)$ be a PE space and $\gamma \subset \del M$ a $\calC^{3,\alpha}$ embedded curve, and
suppose that $Y^2 \subset M$ is a properly embedded minimal surface with asymptotic boundary $\gamma$, an embedded
closed curve in $\del M$. Then the renormalized area $\calA$ of $Y$ is equal to
\begin{equation}
\label{formula} \calA(Y) = -2\pi \chi(Y) - \frac{1}{2}\int_Y|\whk|^2 \, dA + \int_Y W_{1212}\, dA,
\end{equation}
where $\whk$ is the trace-free second fundamental form of $Y$ and $W_{1212}$ is the Weyl curvature of $g$ evaluated
on any orthonormal basis for $TY$.  In particular, the integrals on the right are  convergent. If $Y$ is any
properly embedded surface which extends to be a $\calC^2$ surface with boundary in $\overline{M}$ intersecting
the boundary orthogonally, then (with the convention that $H = (\tr{k})/2$), the renormalized area is equal to
\[
\calA(Y) = -2\pi \chi(Y) + \frac{1}{2}\int_Y \left(2|H|^2- |\whk|^2\right) \, dA + \int_Y W_{1212}\, dA.
\]
\end{proposition}

\begin{proof}
We begin with some preliminary observations and calculations.

First, denote by $R_{ijk\ell}$ and $(R_Y)_{ijk\ell}$ the
components of the curvature tensor of $g$ and of the induced
metric on $Y$, respectively. The Ricci curvature of $g$ satisfies
$R_{ij} = -n g_{ij}$, and from the standard decomposition of the
curvature tensor of an Einstein metric, the components of the Weyl
tensor for $g$ are given by
\begin{equation}
W_{ijk\ell} = R_{ijk\ell} +g_{ik}g_{j\ell}-g_{i\ell}g_{jk}.
\label{rasm}
\end{equation}
Fix a point $p \in Y$ and choose an oriented orthonormal basis
$\{e^1, \ldots, e^{n+1}\}$ for $T_pM$ such that $e^1$ and $e^2$
are an oriented basis for $T_pY$.  Now, denoting by $k_{ij}^s$,
$i,j = 1,2$, $s=3, \ldots, n+1$, the components of the second
fundamental form of $Y$ at $p$, the Gauss-Codazzi equations become
\[
\label{yela} R_{1212}=(R_{Y})_{1212} - \sum_{s=3}^{n+1}(k_{11}^s
k^s_{22}-k^s_{12}k^s_{12}) = (R_{Y})_{1212} - |H|^2+
\frac{1}{2}|\whk|^2.
\]
To check this last equality, simply note that for each $s$,
$k^s_{11}+k^s_{22}= 2H^s$, so $k^s_{ii} = \whk^s_{ii} + H^s$ and
$k^s_{ij} = \whk^s_{ij}$ for $i \neq j$, and hence
\[
\begin{array}{rcl}
\sum_{s=3}^{n+1}\left( k^s_{11}k^s_{22}-k^s_{12}k^s_{12}\right) &=
& \sum_{s=3}^{n+1} \left((\whk^s_{11} + H^s)(\whk^s_{22} + H^s)
-(\whk^s_{12})^2\right) \\ & = & \sum_{s=3}^{n+1} (H^s)^2 -
\frac{1}{2}((\whk_{11}^s)^2 + (\whk^s_{22})^2 + 2
(\whk^s_{12})^2).
\end{array}
\]

Combined with (\ref{rasm}), this gives
\begin{equation}
\label{manage} (R_Y)_{1212} + \frac{1}{2}|\hat{k}|^2 -|H|^2 -
W_{1212} = -1.
\end{equation}
This equation holds at each point. The first term on the left is
simply the Gauss curvature $K$ of $Y$; for simplicity, we continue
to write $W_{1212}$ for the third term on the left, noting that it
is independent of orthonormal frame.

Now integrate over $Y_\e = Y \cap \{x \geq \e\}$ to obtain
\[
\int_{Y_\e}K\, dA   - \frac{1}{2}\int_{Y_\e}\left(
2|H|^2 - |\hat{k}|^2\right)\, dA - \int_{Y_\e} W_{1212}\, dA =
-\int_{Y_\e} dA.
\]
By the Gauss-Bonnet theorem, since $\chi(Y_\e) = \chi(Y)$ for $\e$ small enough,
\[
\int_{Y_\e} K\, dA = 2\pi \chi(Y)-\int_{\gamma_\e}\kappa_\e \, ds,
\]
where $\kappa_\e$ is the geodesic curvature of the boundary $\gamma_\e := \partial Y_\e$ in
$Y_\e$ and $ds$ is the length element with respect to the metric induced by $g$. Altogether we get
\[
\int_{Y_\e} dA = -2\pi \chi(Y)+
\int_{\gamma_\e}\kappa_\e \, ds +\frac{1}{2}\int_{Y_\e}\left(2|H|^2 - |\whk|^2\right) \, dA +
\int_{Y_\e} W_{1212}\, dA.
\]

To proceed further, we use the formula that as $\e \searrow 0$,
\begin{equation}
\label{geodcurv} \int_{\gamma_\e}\kappa_\e\, ds= \frac{\mbox{length}\, (\gamma_\e)}{\e}+ \calO(\e).
\end{equation}
Deferring the proof of this for a moment, using the basic definition of renormalized area via Hadamard regularization
in (\ref{eq:defrenarea}), we find that
\begin{equation}
\label{qasir3} \calA(Y) = -2\pi \chi(Y) + \lim_{\e \to 0} \left(\frac{1}{2}\int_{Y_\e}
\left( 2|H|^2 - |\whk|^2\right)\, dA + \int_{Y_\e} W_{1212}\, dA\right).
\end{equation}
In order to show that the second and third terms on the right have limits as $\e \searrow 0$, recall the
transformation law
\[
\whk_{ij}(e^{2\phi}g)= e^{\phi}\whk_{ij}(g),
\]
for the trace-free second fundamental form $\whk(g)$ under the conformal
change of ambient metric from $g$ to $e^{2\phi}g$ (this is true no matter the dimension or
codimension of the submanifold $Y$). When $\dim Y = 2$,
\[
|\whk(e^{2\phi}g)|^2_{e^{2\phi} g} \, dA_{e^{2\phi}g}  = |\whk(g)|^2_g\, dA_g.
\]
Similarly, the components of the Weyl tensor transform as
\[
W_{ijk\ell}(e^{2\phi}g) = e^{-2\phi}W_{ijk\ell}(g) \Longrightarrow
W_{1212}(e^{2\phi}g) \, dA_{e^{2\phi}g} = W_{1212}(g)\, dA_g.
\]
Thus these two potentially worrisome terms do have a limit. Similarly, even when $Y$ is not minimal, by the
calculations in \S 2, the mean curvature $H$ is $\calO(x)$, so its integral has a limit too.

It remains to prove (\ref{geodcurv}). Denote by $\overline{\kappa}_\e$ the geodesic curvature of $\gamma_\epsilon$
with respect to the metric $\olg = x^2g$ and $\overline{n}$ the interior $\olg$-unit normal to $\del Y_\e$ in $Y_\e$.
Since $u \sim u_2 x^2+u_3 x^3+\calO(x^{3+\alpha})$, it follows that $\overline{n}=
(1+\calO(x^2))\del_x+V$, where $\olg(V,\del_x) = 0$. Now, geodesic curvature also transforms nicely under conformal
re-scalings: $\kappa_\e=\e (\overline{\kappa}_\e+\del_{\overline{n}}\, \log x)$. Since $\olg(F_{ss},\del_x) = 0$
at $x=0$, we deduce $\overline{\kappa}_\e=\calO(\e)$, hence $\kappa_\e=1+\calO(\e^2)$; recalling too
that $ds = \e^{-1}d\overline{s}$, we obtain finally
\[
\int_{\gamma_\e} \kappa_\e\, ds = \frac{\mbox{length}\,(\gamma_\e)}{\e} + \calO(\e),
\]
as claimed.
\end{proof}

The expression $\frac12 \int_{Y_\e} (2H^2 -  |\whk|^2)\, dA$ is finite only when $Y$ intersects $\del M$ orthogonally.
Indeed, using the notation and formul\ae\ from \S 2.3 again, suppose
that $Y$ is written as a normal graph over the vertical cylinder $\Gamma$ over the boundary curve $\gamma$, but
do not assume that $u_x(s,0) \equiv 0$. Now $H= 2 u_x(s,0)+\calO(x)$, as follows from the formul\ae\
$H =x \big( \overline{H}+ 2\del_{\olnu}(\log x)\big)$ and $\del_{\olnu}x = \olg(\olnu, \del_x) = u_x(s,0) + \calO(x)$.
Recalling again that $|\whk|^2\, dA$ is conformally invariant, we see that
\[
\frac12 \int_{Y_\e} (2H^2 -  |\whk|^2)\, dA = \frac{4|u_x(s,0)|^2}{\e} + \calO(\log \e)
\]
does not have a limit as $\e \searrow 0$ unless $u_x(s,0) \equiv 0$.  This is consistent with the fact that
the definition of renormalized area  $\calA(Y)$ via Hadamard regularization is independent of choice of special
bdf $x$ only when this same condition is satisfied.

\section{The moduli spaces $\calM_k(M)$}
Fix the convex cocompact hyperbolic $3$-manifold $(M,g)$ and an integer $k \geq 0$. We define
$\widetilde{\calM}_k(M)$ to be the space of all properly embedded surfaces of genus $k$
which extend to $\overline{M}$ as $\calC^{3,\alpha}$ submanifolds with boundary and which intersect
$\del M$ orthogonally, and $\calM_k(M)$ the subspace of all such surfaces which are minimal.
In this section we study the structure of these moduli spaces, which are the natural domains for the
renormalized area functional, as well as some properties of the natural map $\Pi$ which assigns to any such $Y$
its asymptotic boundary $\del Y = \gamma$, which is a $\calC^{3,\alpha}$ closed (but possibly disconnected)
embedded curve in $\del M$.

It is a standard fact that the space of all $\calC^{3,\alpha}$ surfaces $Y\subset M$ with boundary $\gamma$
lying in $\del M$ is a Banach manifold. The space $\widetilde{\calM}_k(M)$ defined above is clearly a closed
submanifold (of infinite codimension). Similarly, the space $\calE$ of all $\calC^{3,\alpha}$ closed embedded (but
not necessarily connected) curves $\gamma \subset \del M$ is also a Banach manifold.
The corresponding structure for the smaller space of minimal surfaces is also true.
\begin{proposition}
For each $k$, $\calM_k(M)$ is a Banach manifold.
\end{proposition}
\begin{proof}
Fix any $Y \in \calM_k(M)$ and assume for the moment that $\del Y = \gamma$ is actually a $\calC^\infty$ embedded
curve in $\del M$. We construct a coordinate neighbourhood around $Y$ in $\calM_k$ which in the generic (nondegenerate)
setting is modelled on a small ball around $0$ in the space of Jacobi fields for the minimal surface operator on
$Y$ which are $\calC^{3,\alpha}$ up to $\del Y$; this ball in turn is identified with a small ball in the space of
$\calC^{3,\alpha}$ normal vector fields along $\gamma$. We make this nondegeneracy condition explicit below.

To set this up, let $\nu$ be the unit normal (with some fixed choice of orientation) along $Y$. If $\phi$ is
any scalar function on $Y$ which is small in $\calC^{3,\alpha}$, we can define a new surface
\[
Y_{0,\phi} = \{ \exp_{p}(\phi(p)\nu(p)): p \in Y\},
\]
which we call a normal graph over $Y$.

The mean curvature of $Y_{0,\phi}$ is computed by a nonlinear elliptic second order operator $\calF(\phi)$. The precise
expression of this operator is rather complicated, but its linearization has the familiar form
\[
\left. D\calF\right|_{\phi=0} := L_Y = \Delta_{Y} + |A_{Y}|^2 - 2;
\]
here $A_{Y}$ is the second fundamental form of $Y$ and $\Delta_{Y}$ is its Laplacian with respect to the induced metric.

This Jacobi operator, $L_Y$, is an elliptic uniformly degenerate operator of order $2$. Its normal operator is
\[
N(L_Y) = t^2 \del_t^2 + t^2 \del_v^2 - 2
\]
since the second fundamental form $A_Y$ vanishes at $\del Y$; the leading (second order) term is just the Laplacian
on the hyperbolic plane, so $N(L_Y) = \Delta_{\HH^2} - 2$. The indicial roots are $\mu_1 = -1$, $\mu_2 = 2$,
hence solutions of $L_Y u = 0$ satisfy $u \sim a(y)x^{-1}+ \ldots $ or $u \sim a(y)x^2 + \ldots$.
Note that since the $g$- and $\overline{g}$-unit normals are related by $\nu = x \olnu$, in the case
where $u$ blows up as $x \searrow 0$, the product $u \nu = (xu)\olnu$ behaves like $(a(y) + \calO(x))\olnu$,
or in other words, the solutions growing at this rate are the ones which are bounded (but not blowing up)
at $x=0$ with respect to $\overline{g}$, and hence correspond to moving the boundary curve $\gamma$
nontrivially. In this $\overline{g}$ normalization, the decaying Jacobi fields vanish like $x^3$,
which should be no surprise.   In any case, it follows directly from
self-adjointness and integration by parts that
\[
N(L_Y): t^\mu H_0^2(\HH^2; t^{-2} dtdv) = t^{1+\mu} \,
H_0^2(\HH^2; dtdu) \longrightarrow t^{1+\mu}\, L^2(\HH^2; dtdu)
\]
is invertible when $\mu = 0$. By Proposition \ref{pr:eeo}, this is true for any $-1 < \mu < 2$ and for any
$\mu$ in this range,
\begin{equation}
L_Y: x^\mu \Lambda^{2,\alpha}_0 \longrightarrow x^\mu \Lambda^{0,\alpha}_0
\label{eq:Jacws}
\end{equation}
is Fredholm of index zero. We call the minimal surface $Y$ {\it nondegenerate} if the nullspace $K_\mu$
of this mapping contains only $0$ for any $\mu \in (-1,2)$; in this case, (\ref{eq:Jacws}) is surjective.
In general, its cokernel is canonically identified with $K_\mu$ in the following sense. First note that by Proposition
\ref{pr:eeo} again, $K_\mu \subset x^2\Lambda^{2,\alpha}_0(Y)$ (indeed, if $\gamma$ is smooth, any $u \in K_\mu$ is
polyhomogeneous, i.e.\ has full tangential regularity),
so we may as well drop the subscript $\mu$. Next, if $f \in x^\mu \Lambda^{0,\alpha}_0(Y)$ lies in the range of
(\ref{eq:Jacws}), $f = L_Y w$, then obviously $\int_Y u f \, dA_Y = \int_Y u L_Y w \, dA_Y = 0$ for all $u \in K$.
(Note that this integral makes sense since $\mu > -1$.)  However, this gives precisely the correct number
of linear conditions, so this necessary condition is also sufficient.

To study $\calM_k(M)$, we must consider a broader class of deformations of $Y$ where the boundary curve $\gamma$ also varies.
Let $\olnu = x^{-1}\nu$ be the unit normal to $Y$ with respect to the conformally compactified metric $\olg = x^2 g$.
This vector field extends smoothly to $\overline{Y}$, and its restriction to $\gamma = \del Y$ is the unit normal
$\olN$ to this curve in $\del M$ with respect to $h_0$. Any nearby curve can be written as a normal graph
\[
\gamma_\psi = \{\exp_p(\psi(p) \olN(p)) : p \in \gamma\}
\]
(where now $\exp$ is with respect to $h_0$). We now define an extension operator $\calE$ which assigns to
any small $\psi$ a surface $Y_{\psi,0}$ which is `approximately minimal' and which has $\del Y_{\psi,0} = \gamma_\psi$.
To do this, let $u$ be the graph function for $Y$ over the cylinder $\Gamma$. We define a new graph function $u_\psi$
in some neighbourhood $\{x < \e\}$ of the boundary such that $u_\psi(s,0) = \psi(s)$, and $\del_x^j u_\psi(s,0)$, $j = 1,2$
is determined by the formal expansion of solutions for $\calF$; $\del_x^3 u_\psi(s,0)$ could be chosen freely,
but we set it equal to $u_3(s)$. Now let $U_\psi = \chi u_\psi + (1-\chi)u$ where
$\chi $ is a cutoff function which equals $1$ near $x=0$. It is not hard to check that $\calF(U_\psi)
\in x^\mu \Lambda^{1,\alpha}$ for some $0 < \mu < 2$. The extension $\calE$ can be chosen to depend smoothly
on $\psi$. We then have that
\[
\left.D\calE\right|_{0}(\hat{\psi}) = w
\]
is a function on $Y$ which satisfies $w \sim x^{-1}\hat{\psi}$ as $x \searrow 0$ and $L_Y w = \calO(x^\mu)$
for some $\mu \in (0,2)$.

Finally, perturb $Y_{\psi,0}$ to a normal graph over it using the unit normal for $Y_{\psi,0}$ and as graph function
any small $\phi \in x^\mu \Lambda^{2,\alpha}_0(Y)$. The resulting surface will be denoted $Y_{\psi,\phi}$, and we
write its mean curvature as $\calF(\psi,\phi)$. Thus if
${\mathcal B}$ is a small neighbourhood of the origin in $\calC^{3,\alpha}(\gamma) \times x^\mu \Lambda^{3,\alpha}_0(Y)$, then
\begin{equation}
\calF: {\mathcal B} \longrightarrow x^\mu \Lambda^{1,\alpha}_0(Y)
\label{eq:nlpm}
\end{equation}
is a smooth mapping.

A neighbourhood of $Y$ in $\calM_k(M)$ is identified with the space of solutions to $\calF(\psi,\phi) = 0$, and so may
be studied by the implicit function theorem. Note that
\[
\left. D\calF \right|_{(0,0)} (\hat{\psi},\hat{\phi}) = L_Y( D\calE(\hat{\psi}) + \phi).
\]
When $Y$ is nondegenerate, $D_2\calF|_{(0,0)} = L_Y$ on $x^\mu \Lambda^{2,\alpha}_0(Y)$ is already surjective;
this yields the existence of a smooth map $\calG$ defined in a neighbourhood of $0$ in $\calC^{3,\alpha}(\gamma)$ to
$x^\mu \Lambda^{2,\alpha}_0(Y)$ such that $\calF(\psi,\calG(\psi)) \equiv 0$, and so that all elements of
the nullspace of $\calF$ near $(0,0)$ are of this form.

In the degenerate case, we must show that by allowing $\hat{\psi}$ to vary over some suitable finite dimensional 
subspace of infinitesimal deformations of $\gamma$,
we can still obtain a surjective map. If this were to fail, then there would exist a nontrivial $u \in K$ such that
for all $\hat{\psi}$ and $\hat{\phi}$, $L_Y(D\calE(\hat{\psi}) + \hat{\phi})\perp u$. Write $\eta = D\calE(\hat{\psi})$. Then
\[
0 = \int_Y L_Y(\eta + \hat{\phi}) u  = \int_{\gamma} n \cdot \nabla(\eta + \hat{\phi}) u - (\eta + \hat{\phi}) n \cdot \nabla u
= - 2 \int_\gamma \hat{\psi} u_0,
\]
which implies that $u_0$ (the leading coefficient of $x^2$ in the expansion of $u$) is orthogonal to every $\hat{\psi}$,
which is impossible. This proves that $\calF$ is always surjective as a function of both $(\psi,\phi)$, and hence
finally that $\calM_k(M)$ is a smooth Banach manifold in a neighbourhood of $Y$.
\end{proof}

\begin{proposition}
The natural map
\[
\Pi: \calM_k(M) \longrightarrow \calE(\del M)
\]
given by $\Pi(Y) = \del Y$ is Fredholm with index $0$.
\end{proposition}
(Recall that this means that if $Y \in \calM_k(M)$, then $D\Pi_{Y}$ is a Fredholm map from $T_Y \calM_k(M)$ to
$T_{\Pi(Y)}\Gamma$ with the dimensions of its kernel and cokernel equal to one another.)
\begin{proof}
Let $K$ denote the nullspace of $L_Y$ acting on functions of the form $D\calE|_0(\hat{\psi}) + \hat{\phi}$,
$\hat{\psi} \in \calC^{3,\alpha}(\gamma)$, $\hat{\psi} \in x^\mu \Lambda^{2,\alpha}_0(Y)$. If $Y$ is nondegenerate,
then $K$ contains no elements of the form $(0,\hat{\phi})$, so $D\Pi_Y$ is an isomorphism. If $Y$ is
degenerate, however, then the proof above shows that if $\ell = \dim(K \cap x^\mu\Lambda^{2,\alpha}_0(Y))$,
so that $\dim \ker D\Pi_Y = \ell$, then we can make $L_Y$ surjective by supplementing $x^\mu \Lambda^{2,\alpha}_0(Y)$
with an $\ell$-dimensional space $H$ of functions of the form $D\calE_0(\hat{\psi})$ (and we may even assume that
each $\hat{\psi}$ is $\calC^\infty$). Let $H'$ be any choice of complement of $H$ in
$\calC^{3,\alpha}(\gamma)$. The implicit function function theorem shows that there exists a smooth map $\calG$
from $H'$ to $H \oplus x^\mu\Lambda^{3,\alpha}_0(Y)$ such that $\calF(\hat{\psi},\calG(\hat{\psi})) \equiv 0$,
so the codimension of the range of $D\Pi_Y$ is $\ell$ too. Hence $\mbox{ind}\,(D\Pi_Y) = 0$, as claimed.
\end{proof}

The final general result about these moduli spaces is contained in the
\begin{proposition}
$\Pi$ is a proper mapping.
\label{pr:proper}
\end{proposition}
\begin{proof}
We must show that if $\gamma_j$ is a sequence of elements in $\calE$ such that $\gamma_j \to \gamma$ in
$\calC^{3,\alpha}$, and if $Y_j \in \calM_k(M)$ has $\del Y_j = \gamma_j$, then (possibly after passing to
a subsequence) $Y_j$ converges to a properly embedded minimal surface $Y$ with genus $k$ and
$\del Y = \gamma$.

Let $\Gamma$ be the vertical cylinder over $\gamma$, and let $u_j$ be the horizontal graph function
corresponding to the surface $Y_j$. A priori, the function $u_j$ is only defined on some vertical
strip where $x < \e_j$. The first step is to show that $\e_j$ can be chosen independently of $j$.
The only thing which prevents these graphs from existing on a uniform strip would be if the $u_j$
did not have a uniform gradient bound, or in other words, that there exists a sequence
$(s_j,x_j)$ with $s_j$ in the parameter interval for $\gamma$ and $x_j \searrow 0$, and such
that $|\nabla u_j(s_j,x_j)| = 1$, say (any positive number would do), and $|\nabla u_j(s_j,x_j)|< 1$
for all $s$ and for $x < x_j$. The gradient and norm here are with respect to the Euclidean metric.
Perform a hyperbolic rescaling by a factor $\frac{1}{x_j}$, centered at the point
$(\gamma(s_j),0)$, and then a translation and rotation to move $(s_j,0,0)$ to the origin in $\RR^2$ and
to make the rescaled curve $\gamma_j$ tangent to the $y^1$-axis. The result is a minimal surface
$\widetilde{Y}_j$ in the upper half-space, defined in a ball of expanding radius tending to infinity,
which passes through $(0,0,0)$ in the boundary, and which can be expressed as a horizontal graph
$y^2=F_j(y^1,x)$ over some large ball in the vertical $(y^1,x)$-plane. By construction, $F_j(0,1)=
\frac{u_j(s_j,x_j)}{x_j}$; by Rolle's theorem, this is bounded by $|\del_x u_j(s_j,x_j')|$ for
some $x_j' < x_j$, hence by construction $|F_j(0,1)|\leq 1$ and $\frac{\del F_j}{\del x}(0,1)=1/x_j$.

Passing to a subsequence, as $j \to \infty$ this minimal surface converges to a complete
minimal surface $\widetilde{Y} \subset \HH^3$ whose boundary is the limit of rescalings of $\gamma$,
i.e.\  a straight line, and which can be expressed as a horizontal graph $y^2=F(y^1,x)$ over all of
$\RR_{y^1}\times \RR^+_x$ with $|F(0,1)|\leq 1$. However, by construction, the tangent space of
$\widetilde{Y}$ is not vertical at the point $(0,1,F(0,1))$, which contradicts the fact that the
unique minimal surface in hyperbolic space with boundary a straight line is a totally geodesic plane.

This argument proves that the graph functions $u_j$ are defined on a uniform interval $[0,\e]$, and moreover that
the boundary curves at height $x=\e$ are also converging in $\calC^{3,\alpha}$ (in fact, in $\calC^\infty$ by
interior elliptic estimates). Notice that this already proves that no handles can slide off to infinity,
provided the boundary curves remain uniformly smooth enough. Let $Y_{j,\e} = Y_j \cap \{x \geq \e\}$. This is now
a sequence of compact minimal surfaces with boundary in the convex set $\{x \geq \e\} \subset M$. The proof will
be finished if we can prove that these surfaces have a convergent subsequence. This in turn follows from the
results of Anderson \cite{An3} and White \cite{Wh}. In order to apply their results, it suffices to show
that the genera of the $Y_{j,\e}$ remain bounded, which is obvious by definition, and that the areas of these
surfaces are also bounded. This follows from the Gauss-Bonnet theorem: since each $Y_j$ is minimal,
its Gauss curvature satisfies $K \leq -1$, which implies that
\[
\mbox{Area}\,(Y_{j,\e}) \leq \int_{Y_{j,\e}} (-K)\, dA = -2\pi \chi(Y_j) + \int_{\del Y_{j,\e}} \kappa\, ds.
\]
Now, the first term is fixed, so we must show that the second term is bounded. But this is immediate
from standard elliptic estimates applied to the graph function $u_j$ for $Y_j$ in the
annulus $\e/2 < x < 2\e$ since $\e$ is now fixed.
\end{proof}

As explained in \cite{W}, it is important to work with a slightly different regularity condition: we shall
replace $\calC^{k,\alpha}$ by the closure in this space of  $\calC^\infty$. This smaller subspace is separable,
whereas $\calC^{k,\alpha}$ is not, so with this new regularity restriction (which we shall not comment
on further) both $\widetilde{\calM}_k(M)$ and $\calE(\del M)$ are separable Banach manifolds.

Using all of these facts, we may now define the degree of $\Pi$ by
\[
\mbox{deg}(\Pi) = \sum_{Y \in \Pi^{-1}(\gamma)}  (-1)^{n(Y)},
\]
where $\gamma$ is a regular value of $\Pi$, so each $Y \in \Pi^{-1}(\gamma)$ is nondegenerate, and
where $n(Y)$ denotes the number of negative eigenvalues of $-L_Y$. This degree is a well-defined invariant
on each component of $\calE(\del M)$ (once we have fixed the integer $k$ and the component of
$\calM_k(M)$ mapping to that isotopy class of boundary curves).

For example, when $M = \HH^3$, and $\gamma$ is any convex curve, then by the maximum principle, there
is exactly one properly embedded minimal surface $Y$ with $\del Y = \gamma$,
and necessarily, its genus is $0$. This proves that when $M$ is the entire hyperbolic space, then
$\mbox{deg}\, (\Pi_0) = 1$ while $\mbox{deg}\, (\Pi_k) = 0$ for $k > 0$, on the component of
$\calE$ containing connected curves. This has some interesting consequences. For
example, Anderson \cite{An2} displayed a connected curve which bounds a minimal surface
of genus $k > 0$; by genericity, we can assume that this curve is regular for $\Pi$,
and since the degree equals zero, we obtain the existence of yet another element in
$\calM_k(\HH^3)$ with boundary equal to this same curve.  On the other hand, de Oliveira and Soret \cite{OS}
construct {\it stable} properly embedded minimal surfaces in $\HH^3$ with arbitrary genus, where
the boundary curve has any prescribed number of components. Here too, for any given boundary curve,
we conclude the existence of at least one other element of $\calM_k$ with that boundary curve
which is unstable. It would be interesting to compute the degree of $\Pi$ precisely in some of these
other cases.

\section{Area minimization and renormalized area}
We now investigate the role of locally area minimizing surfaces in the study of renormalized area
for minimal and nonminimal properly embedded surfaces.

\subsection{Renormalized area of absolute minimizers}
\begin{proposition}
\label{concl1} Let $\gamma$ be a $\calC^{3,\alpha}$ embedded curve
in $\del M$ which bounds in $M$. Suppose that $Y_1$ and $Y_2$ are
two properly embedded minimal surfaces with $\del Y_1 = \del Y_2 =
\gamma$. If $Y_1$ is area minimizing in $M$, then $\calA(Y_1) \leq
\calA(Y_2)$, and equality holds if and only if $Y_2$ is also an
area minimizer.
\end{proposition}

\begin{proof}
Fix a special boundary defining function $x$ and set $Y_{j,\e} = Y_j \cap \{x \geq \e\}$ and
$\gamma_{j,\e} = \del Y_{j,\e}$. The functions $u_1$ and $u_2$ for these two surfaces agree up to
order three, so in terms of any local coordinate $s$ on $\gamma$, $|u_1(s,x) - u_2(s,x)| \leq C x^3$, with
corresponding estimates for the first $3$ derivatives. If $S_\e$ denotes the region between $\gamma_{1,\e}$
and $\gamma_{2,\e}$ in $\{x=\e\}$, this gives
\[
\mbox{Area}\,(S_\e) = \calO(\e).
\]

Recalling that $\mbox{Area}\,(Y) = \e^{-1}\mbox{length}\,(\del Y) + \calA(Y) + \calO(\e)$, we obtain
\[
\mbox{Area}\,(Y_{2,\e}) - \mbox{Area}\,(Y_{1,\e}) = \calA(Y_2) - \calA(Y_1) + \calO(\e),
\]
since $\del Y_1 = \del Y_2$.

Thus if $\calA(Y_1) > \calA(Y_2)$, then the new (nonsmooth) surface $Y_{1,\e}' = Y_{2,\e} \cup S_\e$ would eventually
have area smaller than $Y_{1,\e}$. Indeed, by the inequalities above,
\[
\mbox{Area}\,(Y_{1,\e}') \leq \mbox{Area}\,(Y_{2,\e}) + C_1\e \leq \mbox{Area}\,(Y_{1,\e}) +
\calA(Y_2) - \calA(Y_1) + C_2 \e < \mbox{Area}\,(Y_{1,\e})
\]
for $\e$ small enough. This contradicts the fact that $Y_1$ is area-minimizing.
\end{proof}

\subsection{The renormalized area spectrum}
The result in this last subsection suggests the consideration of the set-valued function
\[
\calE \ni \gamma \overset{\calS_k}{\longmapsto} \{\calA(Y): Y \in \calM_k(M)\ \ \del Y = \gamma\},
\]
which we call the renormalized area spectrum (of degree $k$). Properness of the boundary map
ensures that $\calS_k(\gamma)$ is always a compact set. Proposition \ref{concl1} implies that
if $Y_0$ is any absolutely area minimizing surface with $\del Y_0 = \gamma$, then
$\underline{\calA}(\gamma) :=\calA(Y_0)$ is a lower bound for $\calS_k(\gamma)$ for
{\it every} $k$. On the other hand, trivially by (\ref{malaki}), this set is bounded above
by $-2\pi \chi(Y) = 2\pi (2k + \ell - 2)$, where $\ell$ is the number of components of $\gamma$.
In other words,
\begin{equation}
\calS_k(\gamma) \subset [\underline{\calA}(\gamma), 2\pi(2k+\ell-2) ].
\label{eq:limras}
\end{equation}
Note furthermore that the upper limit is never attained unless there exists a totally geodesic
minimal surface $Y$ with $\del Y = \gamma$, which never happens unless $\gamma$ is a round circle.
The lower bound is attained if and only if there exists a genus $k$ absolute area minimizer with boundary
$\gamma$.

\subsection{Minimizers of the extended renormalized area functional}
In section 4.2 of \cite{ht:assee}, Hirata and Takayanagi assert that minimizers of the extended renormalized area
functional amongst all surfaces with a given boundary are necessarily area minimizing surfaces with this same 
asymptotic boundary. Furthermore, Polyakov has communicated to us his suggestion (based on his work 
\cite{p:fss}) that the area-minimizers amongst all surfaces with a given boundary and a given genus must also 
minimize a functional that depepnds on a quadratic expression of the extrinsic curvature. Both these assertions
 follow easily from our techniques. Recall that we are denoting the extension of $\calA$ to this setting by $\calR$.

\begin{proposition} Let $\gamma$ be a $\calC^{3,\alpha}$ closed curve in $\del M$ which bounds in $M$.  Then
the infimum of $\calR(Y)$ where $Y$ ranges over the set of all $\calC^{3,\alpha}$ surfaces with $\del Y = \gamma$
which intersect $\del M$ orthogonally is attained only by absolutely area-minimizing surfaces.
\end{proposition}

\begin{proof}
Let $Y$ be any $\calC^{3,\alpha}$ surface which intersects $\del M$ at $\del Y = \gamma$. We must show that
if $\calR(Y) \leq \calR(Y')$ for all other such surfaces $Y'$, then $Y$ is absolutely area-minimizing.
Fix a decreasing sequence $\e_j \searrow 0$ and let $Y_j = Y \cap \{x \geq \e_j\}$ and $\gamma_j = \del Y_j$.
Let $Y_j'$ denote an area minizing surface with $\del Y_j' = \gamma_j$. Following the original existence
proof by Anderson, possibly after passing to a subsequence, we may assume that
\begin{equation}
\label{converges}
Y_j'\rightarrow Y'
\end{equation}
where $Y'$ is properly embedded and area-minimizing with $\partial Y'=\gamma$. Standard results imply
that the convergence is in $\calC^\infty$ in the interior, and also $\calC^{3,\alpha}$ up
to the boundary. This may be proved by an argument very similar to that used in establishing
properness in Proposition \ref{pr:proper}. Indeed, if the convergence were not $\calC^1$,
we would be able to take a suitable rescaling and obtain a limiting surface which has boundary
a straight line in $\RR^2 = \del \HH^3$ but which is not a totally geodesic plane, which would
be a contradiction. Knowing this, if the convergence were not $\calC^{3,\alpha}$, we could again
rescale and extract a limiting surface which converges to a totally geodesic plane, but not
smoothly, which also contradicts standard convergence results for minimal surfaces.

Denote the genus of $Y'$ by $k$, so $Y'$ minimizes $\calA$ in $\cup_{s} \calM_s(M)$ amongst all minimal
surfaces with the same boundary and of arbitrary genus.

We now claim that
\begin{equation}
\label{compare} \mbox{Area}\,(Y_j) \geq \mbox{Area}\,(Y_j') =
\frac{1}{\e_j} \mbox{length}(\gamma)+ \calA(Y') +o (1).
\end{equation}
The first inequality is by definition. As for the second equality, we proceed in steps. Since 
$Y'_j\to Y'$ in $\calC^{3,\alpha}$, there exists $\rho>0$ so that each of the annuli $Y_j' 
\setminus Y'_{j,\rho} := Y'_j\cap \{x \leq \rho\}$ and $Y'\setminus Y'_\rho :=  Y' \cap \{x\leq\rho\}$ 
is a normal graph over some portion of the vertical cylinder $\Gamma$ of height $\rho$ above $\gamma$. 
In particular, suppose that $Y'_j\setminus Y'_{j,\rho}$ is the graph of a function $u_j$, defined on 
some band $\Gamma_j = \{\epsilon_j \leq x \leq \rho\}$ and $Y' \setminus Y'_\rho$ is the
graph of a function $u$ defined on the band $\{ 0\le x\leq\rho\}$. Clearly the areas of the portions 
where $x \geq \rho$ converge, i.e.\ $\mbox{Area}\,(Y'_{j,\rho}) \to \mbox{Area}\,(Y'_\rho)$. Thus if
we denote by $Y'{j,\rho,\e_j}$ the portion of $Y'$ that lies below the hyperplane $x=\rho$ and above
the hyperplane $x=\epsilon_j$ it suffices to show that
\[
\lim_{j \to \infty}\left(\mbox{Area}\, (Y'_{j,\rho,\e_j} -\mbox{Area}\, (Y'_j \setminus Y'_{j,\rho}) \right) = 0.
\]
This follows by direct computation. First write $u_j(x,s)=x^2\tilde{u}_j(x,s)$, $u(x,s)=x^2\tilde{u}(x,s)$, so that
$\tilde{u}_j\to \tilde{u}$ in $\calC^1$. Denote by $J_u, J_{u_j}$ the Jacobians for parametrizations of these
surfaces in the $(s,x)$ coordinate charts as in \S \ref{vms}. Using the formul\ae\ from that section we compute
that, with constants independent of $j$,
\begin{equation}
\label{goettingen}
\begin{split}
&|\left(\mbox{Area}\, (Y'_{j,\rho,\e_j} -\mbox{Area}\, (Y'_j \setminus Y'_{j,\rho}) \right)| =
\left|\int_{\e_j}^\rho\int_\gamma \frac{J_u-J_{u_j}}{x^2}\, dsdx \right|
\\ &\leq C \, \int_{\e_j}^\rho \frac{1}{x^2}\int_\gamma \left| u_s^2-u_{j,s}^2-\left((1-\kappa u)^2(1+u_x^2)
-(1-\kappa u_j^2)(1+u_{j,x}^2)\right)\right|\, dsdx
\\ &\le C'\int_{\e_j}^\rho\int_\gamma|\tilde{u}_j-\tilde{u}|+|\tilde{u}_{j,s}-\tilde{u}_s|+
|\tilde{u}_{j,x}-\tilde{u}_x|\, dsdx \to 0
\end{split}
\end{equation}
as $j \to \infty$. This proves (\ref{compare}).
Subtracting $\e_j^{-1}\mbox{length}\,(\gamma)$ from (\ref{compare}) and passing to the limit gives
\[
\calR(Y) \geq \calR(Y') = \calA(Y').
\]

The second step of the proof is to show that if $\calR(Y)= \calR(Y')$ then $Y$ must be an area-minimizer.
By Proposition \ref{concl1}, it suffices to show that $Y$ is minimal. If it is not, then its mean curvature
$H$ is nonvanishing in some open set. We can then perturb $Y$ locally, in some small ball $Y \cap B_r(p)$,
to a new surface $Y''$ which is also smooth, but has smaller area in this ball. Clearly the areas of the
truncations to $x \geq \e$ satisfy $\mbox{Area}\,(Y''_\e)
< \mbox{Area}\,(Y'_\e)$ for $\e$ small, which in turn implies that $\calR(Y) < \calR(Y')$,
contrary to what we proved above.

\end{proof}

\section{First and second variations of renormalized area}
We now begin the variational analysis of the renormalized area functional $\calA$ on each of the moduli spaces
$\calM_k(M)$, as well as for its extension to the unconstrained spaces $\widetilde{\calM}_k(M)$. The first
variation formula for $\calA$ on $\calM_k$ is formally analogous to the corresponding first variation formula
for the renormalized volume of Poincar\'e-Einstein metrics in even dimensions; this formula appears in a paper by
Anderson \cite{An4}, see also Albin \cite{Alb} for a simpler approach. A formula for the second variation of
renormalized volume does not seem to have been computed in the Poincar\'e-Einstein setting.

\subsection{The first variation}
We shall compute the first variation of $\calA$ at any $Y\in \calM_k(M)$; slightly more generally,
we compute the first variation of the extended functional $\calR$ at any $Y\in \widetilde{M}_k(M)$.
Actually, we compute $D\calR$ only applied to compactly supported perturbations, and show
this is not well-defined for arbitrary variations in $\widetilde{M}_k(M)$.

As before, fix a special boundary defining function $x$ on $M$, and for any $Y \in \calM_k(M)$
let $u_3$ denote the free third order term in the expansion of the graph function $u$ of $Y$
with respect to $x$ over its vertical cylinder.

\begin{theorem}
\label{theformula} Fix any $Y \in \calM_k(M)$ with $\del Y = \gamma$. Let $Y_t$ denote a smooth
one-parameter curve in $\widetilde{\calM}_k(M)$, $\del Y_t = \gamma_t$, with $Y_0=Y$ (so the surface
$Y_0$ is minimal but the surfaces $Y_t$, $t\neq 0$, need not be). Write $Y_t$ as a normal graph
$Y_{\phi}$ over $Y_0 = Y$ via
\[
Y_t = \{ F_t(p) = \exp_p (\phi_t(p)\nu(p)): p \in Y\}
\]
where $\phi_0 = 0$.

Setting $\dot{\phi} = \left. \frac{d\,}{dt} \phi_t \right|_{t=0}$, then $\dot{\phi} \sim
x^{-1}\dot{\phi}_0 + \ldots$, and we have
\[
\left. \frac{d\, }{dt}\right|_{t=0} \calA(Y_t)=-3\int_{\gamma} \dot{\phi}_0 \, u_3\, ds.
\]

On the other hand, if $Y\in\widetilde{\calM_k}(M)$ and $\phi_t=0$ outside some compact set $K\subset Y$, then:
\[
\left. \frac{d\, }{dt}\right|_{t=0} \calA(Y_t)= 2\int_Y H \dot{\phi}dA
\]
\end{theorem}

\begin{proof}
By virtue of (\ref{formula}), we have
\begin{equation}
\label{seismos} \left. \frac{d\,}{dt}\right|_{t=0} \calA
(Y_t)=\frac{1}{2} \left. \frac{d\, }{dt}\right|_{t=0} \left(
\int_{Y_{t}}(\tr k_t)^2- |k_t|^2\, dA_{t} \right),
\end{equation}
where $k_t$ and $dA_t$ denote the second fundamental form and area form on $Y_t$  ($\tr k_t$ denotes the
trace of the second fundamental form of $Y_t$).

In order to compute this, introduce the following notation. Fix any point $p \in Y$
and choose an orthonormal moving frame $\{e_1,e_2\}$ on $Y$ near $p$, as well as a unit
normal vector field $\nu$, so that $e_1$ and $e_2$ are the principal directions and
$\nabla_{e_i}e_j = - \kappa_i \delta_{ij} \nu$ at $p$ (the $\kappa_i$ are the
principal curvatures of $Y$). Define $\eta_i(t) = (F_t)_*(e_i)$, so $\eta_i(0) = e_i$
but $\{\eta_1(t),\eta_2(t)\}$ is no longer orthonormal when $t \neq 0$. We define
\[
g_{ij}(t) = \langle \eta_i(t), \eta_j(t) \rangle \qquad \mbox{and} \qquad
k_{ij}(t) = - \langle \nabla_{\eta_i(t)} \eta_j(t), \nu(t)\rangle,
\]
where $\nu(t)$ is the unit vector orthogonal to both $\eta_1(t)$ and $\eta_2(t)$ with
$\nu(0) = \nu = e_3$. Finally, write $T = F_* \del_t$, so $T = \dot{\phi}\nu$ when $t=0$.

In the following, all terms are to be computed eventually at $p$. We first compute that
\[
\left. \nabla_T \eta_i \right|_{t=0} = \left. \nabla_T \nabla_{\eta_i} F \right|_{t=0} =
\left. \nabla_{e_i} \nabla_T F \right|_{t=0} =
\nabla_{e_i} \left(\dot{\phi}\nu\right) = \dot{\phi}\kappa_i e_i + \dot{\phi}_i \nu,
\]
which yields immediately
\[
\left. T g_{ij}(t) \right|_{t=0} = \langle \dot{\phi}\kappa_i e_i + \dot{\phi}_i \nu, e_j \rangle +
\langle e_i, \dot{\phi}\kappa_j e_j + \dot{\phi}_j \nu \rangle = 2 \dot{\phi} k_{ij},
\]
since $k_{ij} = \kappa_i \delta_{ij}$ in this frame. This also implies that $Tg^{ij}(t)|_{t=0}=-2\dot{\phi}k^{ij}$.
It is also standard that
\[
\left. \frac{d\,}{dt}\right|_{t=0}\, F^*(dA_{t}) = \dot{\phi} (\tr k_{ij})\, dA_0.
\]
Finally,
\[
\left. T k_{ij}(t) \right|_{t=0} = - T \langle \nabla_{\eta_i} \eta_j, \nu \rangle
= \left. - \left(\langle \nabla_{T} \nabla_{\eta_i}\eta_j, \nu\rangle + \langle \nabla_{\eta_i}\eta_j,
\nabla_{T}\nu \rangle\right)\right|_{t=0}.
\]
By assumption, $\nabla_{e_i}e_j$ is orthogonal to $Y$, while $\nabla_\nu \nu$ is tangential,
so the second term on the right vanishes. By definition
of the curvature tensor, the other term equals
\begin{equation*}
\begin{split}
- \langle \nabla_{\eta_i} \nabla_T \eta_j + \nabla_{[T,\eta_i]} \eta_j + R(T, \eta_i)\eta_j, \nu \rangle
= - \langle \nabla_{e_i} (\dot{\phi}\kappa_j e_j + \dot{\phi}_j \nu), \nu \rangle
- \dot{\phi} \langle R(\nu,e_i)e_j,\nu \rangle.
\end{split}
\end{equation*}
Observe that
\[
[T,\eta_i] = \nabla_{T} e_i - \nabla_{e_i} (\dot{\phi}\nu) =
\dot{\phi}\kappa_i e_i + \dot{\phi}_i \nu -
\dot{\phi}_i \nu - \dot{\phi}\kappa_i e_i = 0,
\]
so expanding out yields that
\[
\left. T \, k_{ij} \right|_{t=0} = \dot{\phi}
 \kappa_j^2 \delta_{ij} - \dot{\phi}_{ij} + R_{3i3j} =
-(\nabla^2\dot{\phi})_{ij} + \dot{\phi} (k \circ k)_{ij} + R_{3i3j}.
\]

Putting these formul\ae\ together, we compute that
\begin{equation}
\begin{split}& \left. \frac{d\,}{dt} \right|_{t=0} \,  \frac{1}{2}\int_{Y_t} (\tr
k_t)^2-|k_t|^2\, dA_t = -\int_Y \left( -\nabla^2
\dot{\phi}_{ij}\right) k^{ij}+\Delta\dot{\phi} \,\tr k \,dA \\&
+\int_Y (\tr k)^3-3|k|^2(\tr k)+2\tr(k\circ k\circ k)\, dA+2\int_Y
\dot{\phi}({R_{3i3}}^i\tr k -R_{3i3j}k^{ij}) \, dA.
\end{split}
\end{equation}
(Here $\tr(k\circ k\circ k) = g^{ip}g^{jq}g^{\ell r}k_{ir}k_{pj} k_{q \ell}$.)
A simple calculation using the principal curvature decomposition shows that
\[
(\tr k)^3-3|k|^2(\tr k)+2\tr(k\circ k\circ k)=0,
\]
while since $M$ is hyperbolic,
\[
{R_{3i3}}^i\, \tr k -R_{3i3j}\, k^{ij}=-\tr k.
\]
This proves
\begin{equation}
\left. \frac{d\,}{dt} \right|_{t=0} \,  \frac{1}{2}\int_{Y_t} (\tr
k_t)^2-|k_t|^2\, dA_t = -\int_Y  -\nabla^2 \dot{\phi}_{ij}
k^{ij}+\Delta \dot{\phi}\, \tr k\, dA+ 2\int_Y \dot{\phi}\, \tr k\, dA.
\end{equation}

The next step is to evaluate $\int_Y \left(-\nabla^2 \dot{\phi}_{ij} k^{ij}+ \Delta\,
\dot{\phi} \, \tr k\right)\, dA$. To do this, we integrate by parts on $Y_\e = Y \cap \{x \geq \e\}$ to get
\begin{multline}
- \int_{Y_\e} \left( \langle \nabla^2 \dot{\phi} , k\rangle-\Delta\dot{\phi}\, \tr k\right)\, dA = \\
\int_{Y_\e} \dot{\phi}_i \, (\nabla^j k_{ij}-\nabla^ik_j^{\ j}) \, dA
+ \int_{\gamma_\e} \left(\langle \nabla \dot{\phi}, n \rangle \, \tr k - k(\nabla \dot{\phi}, n)\right)
\, ds, 
\end{multline}
where $\gamma_\e = \del Y_\e$ and $n$ is the $g$-unit normal in $Y_\e$ to $\gamma_\e$.
The contracted Codazzi equation states that
\[
\nabla^jk_{ij}= \nabla_i \tr k+ \Ric_{\nu i},
\]
but in fact $\Ric_{\nu i} = -2 g_{\nu i} = 0$ since the $i$ index refers to a vector tangent to $Y$.
Hence only the boundary terms remain. If $Y\in \widetilde{M}_k(M)$ then since $\dot{\phi}$ is compactly supported,
these boundary terms vanish.  If $Y\in \calM_k(M)$, however, then $\tr k=0$ so the first part of the integrand
in the boundary integral vanishes and we only have to evaluate the second one.

To do this, revert to the $(s,x)$ coordinates introduced in \S 2. In terms of these,
the expression becomes
\[
\int_{\gamma_\e} \dot{\phi}_i g^{ij} k_{j\ell}n^\ell \, ds.
\]
To calculate this, we first note that the unit normal $n = n^1 \del_s + n^2 \del_x$ has
coefficients which satisfy
\[
n^1 = \calO(x^3), \quad n^2 = 1 + \calO(x^2).
\]
Thus we may as well set $\ell = 2$ and $n^2 = 1$. Note also that $g^{ij} = x^2 \olg^{ij}$
and $ds = x^{-1}\, d\overline{s}$, where $\overline{s}$ is arclength on $\gamma$
with respect to $\olg$. (This differs slightly from our earlier convention.) The terms in
the integrand are thus
\[
\left(\dot{\phi}_1 x^2 (\olg^{11} k_{12} + \olg^{12}k_{22}) + \dot{\phi}_2 x^2
(\olg^{21} k_{12} + \olg^{22} k_{22}) \right)\, d\overline{s}.
\]
Finally, recall that $\dot{\phi} \sim x^{-1}\dot{\phi}_0$ and use the expansions for the $k_{ij}$
and $g_{ij}$ to deduce that this reduces to
\[
- 3 \int_{\gamma_\e}  \dot{\phi}_0 u_3\, d\overline{s} + \calO(\e).
\]
This completes the proof.
\end{proof}

From this formula we deduce the
\begin{corollary}
Suppose that $Y \in \calM_k(M)$ is a critical point for $\calA$. If $Y$ is a nondegenerate
element in $\calM_k$ (i.e.\ it has no decaying Jacobi fields), then the coefficient $u_3$
in its boundary expansion vanishes identically. In general, if $Y$ is degenerate, then we may
only conclude that $u_3$ is orthogonal to the finite dimensional space spanned by leading 
coefficients $\psi_0$ of Jacobi fields $\psi \sim \psi_0 x^{-1} +  \psi_1 x^0 + \ldots$ along 
$Y$. This is equivalent to the statement that $u_3$ lies in the (finite dimensional) span of 
all leading coefficients $\psi_3$ of decaying Jacobi fields $\psi \sim \psi_3 x^2 + \ldots$ on $Y$. 

If $Y \in \widetilde{M}_k(M)$ is critical for $\calR$, then $Y$ is necessarily minimal,
and $u_3 = 0$ regardless of whether or not $Y$ is degenerate in $\calM_k(M)$.
\label{co:crit}
\end{corollary}

We remark that it would be quite interesting to know if there exist degenerate
critical points of $\calA$. 

\subsection{The second variation}
We now derive the second variation formula for $\calA$ at a general element $Y \in \calM_{k}$. 
The calculation is more involved than for the first variation formula, but is actually fairly
elementary since it relies primarily on some detailed trigonometric calculations.

Before stating the result, we set up the basic notation. We use the same language as in Theorem \ref{theformula}. 
Let $Y_t$ be any $1$-parameter family of surfaces in $\calM_k(M)$, with corresponding family of boundary curves 
$\gamma_t$ in $\del M$. The calculations are local, so we assume that $M = \HH^3$ and we work in
standard upper half-space coordinates. Let $\olN_t$ denote the inward-pointing unit normal to $\gamma_t$
with respect to the Euclidean metric on $\RR^2$. Let $(s,v)$ be Fermi coordinates around $\gamma_0$, i.e.\ 
corresponding to the map $(s,v) \mapsto \gamma_0(s) + v \olN_0(s)$. Then $\gamma_t$ can be written as
a normal graph, $v = \psi_t(s)$, where $\psi_0 \equiv 0$ and $\dot{\psi} \olN_0$ is the normal variation
vector field. Similarly, introduce Fermi coordinates for the Euclidean metric around $Y_0$, using the
inward unit normal $\olnu$; the cylindrical coordinates $(s,x)$ parametrize $Y_0$, and we write
$Y_t$ as a normal graph $v = \phi_t(s,x)$ using the $\olg$-unit normal $\olnu$. The variation vector
field is $\dot{\phi}\olnu$, so $\dot{\phi}$ is a solution of the Jacobi equation $L_{Y_0}\dot{\phi} = 0$. 
The regularity theory for this equation shows that it has an expansion 
\[
\dot{\phi}(s,x) \sim \dot{\phi}_0(s) + \dot{\phi}_2(s) x^2 + \dot{\phi}_3(s)x^3 + \ldots,
\]
where $\dot{\phi}_0(s) = \dot{\psi}(s)$. If $Y_0$ is nondegenerate, the coefficient of $x^3$
is determined by a global Dirichlet-to-Neumann map in terms of the leading coefficient, which we write as
\[
\dot{\phi}_3 =  \frac{1}{6} \left. \del_x^3 \right|_{x=0} \dot{\phi} = \calD_{Y_0} \dot{\phi}_0.
\]

\begin{theorem}
The second variation of $\calA$ along the family $Y_t$ is given by
\[
\left. D^2\calA \right|_{Y_0} (\dot{\phi},\dot{\phi})  = 
-3 \int_{\gamma_{0}} \dot{\phi}_0 \dot{\phi}_3 \, d\overline{s},
\]
or equivalently, when $Y_0$ is nondegenerate, 
\[
- 3 \int_{\gamma_{0}} \dot{\phi}_0 \calD_{Y_0}\dot{\phi}_0 \, d\overline{s}.
\]
\end{theorem}
\begin{proof} Let us begin by setting up some notation. Write each $Y_t$ as a 
horizontal graph over the vertical cylinder $\Gamma_t$ on the curve $\gamma_t$, i.e.\ via 
$(\gamma_t(s),x) + u_t(s,x)\olN_t$.  Through the remainder of this proof, subscripts refer only 
to $t$-dependence but not derivatives, and we often write $\dot{w}$ for $\del_t w$. Set
\[
u_t(s,x) = \frac12 \kappa_t(s) x^2  + u_{3,t}(s) x^3 + \calO(x^{3+\alpha}).
\]
The unit tangent and normal vectors of $\gamma_t$ are
\[
\overline{T}_t = \frac{1}{\sqrt{(\del_s \phi_t)^2 + (1-\kappa_0 \phi_t)^2}}
\left((1 - \kappa_0 \phi_t)\overline{T}_0 +  \del_s \phi_t \olN_0\right),
\]
\[
\olN_t = \frac{1}{\sqrt{(\del_s \phi_t)^2 + (1-\kappa_0 \phi_t)^2}}\left(
-\del_s \phi_t \overline{T}_0 + (1-\kappa_0 \phi_t) \olN_0\right),
\] 
so the component along $\olN_t$ of the variation vector field $\dot{\phi}\, \olN_0$ is 
\[
\langle \olN_t, \dot{\phi} \olN_0 \rangle = \dot{\phi}_t \frac{1-\kappa_0 \phi_t}{\sqrt{(\del_s \phi_t)^2 + 
(1-\kappa_0 \phi_t)^2}} = \dot{\phi}_t \left(1 + (\del_s \phi_t/(1 - \kappa_0 \phi_t)^2 \right)^{-1/2}.
\]
This is equal to $\dot{\phi}_t (1 + \calO(t^2) )$. Hence the first variation of $\calA$ at $Y_t$ is given by 
\[
\frac{d\,}{dt} \calA(Y_t)  = -3 \int_{\gamma_t} u_{3,t} \dot{\phi}_t (1 + \calO(t^2) )\, d\bar{s}.
\]
Differentiating once more at $t=0$ yields
\[
\left. \frac{d^2\,}{dt^2} \calA(Y_t) \right|_{t=0} = 
-3 \int_{\gamma_0} \left( \dot{u}_{3,0} \, \dot{\phi}_0 + u_{3,0} \ddot{\phi}_0 \right)\, d\bar{s},
\]
The second term in the integrand is an artifact of the parametrization. Indeed, recall that for
the composition of any two functions, $A(s)$ and $s = Y(t)$, $(A(Y(t))'' = A''(Y_0)(Y')^2 + A'(Y_0)Y''$, 
but the Hessian of $A$ corresponds only to the first term on the right. In other words, 
\begin{equation}
\left. D^2 \calA\right|_{Y_0}(\dot{\phi},\dot{\phi}) = -3 \int_{\gamma_0} \dot{u}_{3,0} \dot{\phi}_0 \, d\bar{s}.
\label{eq:2ndderiv}
\end{equation}

The key point is to relate $\dot{\phi}_3$ to $\dot{u}_{3,0}$. The long calculation that follows involves expressing 
$\phi_t$ in terms of $u_t$ and then calculating $\del_t \del_x^3$ at $(x,t) = (0,0)$. The relationship between these 
two terms arises as follows. Fix any $P \in Y_0$ with $x$ coordinate $x(P) = x_0$ sufficiently small. Let $\tilde{\lambda}$ 
be the line emanating from $P \in Y_0$ in the direction $\olnu(P)$. Denote by $\tilde{P}$ and $\tilde{R}$ its points of 
intersection with $\Gamma_t$ and $Y_t$, respectively. Note that assuming $\gamma_0$ is convex near $P$, then moving 
from $P$ along the line $\tilde{\lambda}$, one encounters first $\tilde{P}$ then $\tilde{R}$; we orient coordinates 
so that all quantities are positive in this situation. By definition, $\phi_t(P)$ is the distance from $Y_0$ to $Y_t$ 
along $\tilde{\lambda}$, and it can be decomposed as the sum of two signed distances, from $P$ to $\tilde{P}$ and
from $\tilde{P}$ to $\tilde{R}$, which we call $\tilde{\ell}$ and $\tilde{\chi}$, respectively:
\[
\phi_t = |P\tilde{R}| = |P\tilde{P}| + |\tilde{P}\tilde{R}| = \tilde{\ell} + \tilde{\chi}.
\]
We show how to relate each of the two quantities $\tilde{\ell}$ and $\tilde{\chi}$ to $u_t$ in turn.

First consider $\tilde{\ell}$. Let $x_0$ be the $x$ coordinate of the point $P$ and $\alpha$ the angle between 
$\tilde{\lambda}$ and the plane $x=x_0$. From (\ref{eq:defnu}) the orthogonal projection of $\olnu(P)$ onto this plane 
is spanned by $-\del_s u \, \overline{T}_0 + (1 - \kappa_0 u_0) \olN_0$, and 
\[
\sin \alpha = \frac{\del_x u \, (1-\kappa_0 u_0)}{\sqrt{(\del_s u)^2 + (1-\kappa_0 u_0)^2 (1 + (\del_x u_0)^2)}}
\quad
\Longrightarrow  \quad \alpha = \kappa_0 x + 3 u_{3,0} x^2 + \calO(x^3).
\]
Now let $\lambda^1$ denote the line in the plane $x=x_0$ from $P$ in the direction of this projection 
of $\olnu$, and $P^1 = \lambda^1 \cap \Gamma_t$. Setting $\ell^1 = |PP^1|$ (interpreted as a signed distance), 
since $PP^1 \tilde{P}$ is a right triangle, we have 
\[
\tilde{\ell} = \frac{\ell^1}{\cos \alpha} = \ell^1(1 + \frac12 \alpha^2 + \ldots) 
= \ell^1 (1 + \frac12 \kappa_0^2 x^2 + 3 \kappa_0 u_{3,0} x^3 + \ldots).
\]
We shall prove below that
\begin{equation}
\left.\ell^1\right|_{x=0} = \psi_t, \quad \left.\del_x \ell^1\right|_{x=0} = 0, 
\quad \mbox{and}\qquad \left.\del_x^3 \ell^1\right|_{x=0} = 6u_{3,0} + \calO(t^2),
\label{eq:derivs}
\end{equation}
and hence, after some calculation, 
\begin{equation}
\left. \del_t \del_x^3 \tilde{\ell}\right|_{x=t=0} = 18 \kappa_0 u_{3,0} \dot{\phi}_0.
\label{eq:derivltil}
\end{equation}
To prove (\ref{eq:derivs}), we examine $\ell^1$ more closely. Consider the (very thin) triangle $T$ in the plane $x=x_0$ with 
sides the segments of the line $\lambda^1$, the line $\bar{\lambda}$ emanating from $P$ with direction $\olN_0$, 
and the line tangent to $\gamma_t$ at the point of intersection $\bar{\lambda}\cap \gamma_t$. Let $\ell^2$, 
$\bar{\ell}$ and $\hat{\ell}$ denote the lengths of these sides, respectively, and let $\beta$ and $\omega$ denote the angles 
between $\lambda^1 \cap \bar{\lambda}$ and $\bar{\lambda} \cap \gamma_t$. Note that
\[
\begin{array}{rcl}
\cos \beta & = & \frac{1}{\sqrt{(\del_s u_0)^2 + (1 - \kappa_0 u_0)^2}}
\langle -\del_s u \overline{T}_0 + (1 - \kappa_0 u_0) \olN_0, \olN_0 \rangle \\
& = & \left(1 + ( \del_s u_0/(1-\kappa_0 u_0))^2\right)^{-1/2} = 1 - \frac12 ( \frac12 \del_s \kappa_0 \, x^2 
+ \del_s u_{3,0} \, x^3 + \ldots)^2 + \ldots 
\end{array}
\]
\begin{equation}
\Longrightarrow  \beta = \frac12 \del_s \kappa_0 \, x^2 + \del_s u_{3,0} \, x^3 + \ldots.
\label{eq:beta}
\end{equation}
In addition, since $\pi/2 - \omega$ is the angle between $\olN_0$ and $\olN_t$, 
\[
\cos (\frac{\pi}2 - \omega) = 1 - \frac12(\del_s \phi_t + \calO(t^2))^2 + \ldots \Rightarrow
\omega = \frac{\pi}2 - \del_s \phi_t + \calO(t^2).
\]

The law of sines for this triangle is the set of equalities
\[
\frac{\hat{\ell}}{\sin \beta} = \frac{\ell^2}{\sin \omega} = \frac{\bar{\ell}}{\sin(\pi-\beta-\omega)}.
\]
Since $\omega$ is bounded away from $0$, the first equality and (\ref{eq:beta}) imply that $\hat{\ell} = \calO(x^2)$, 
which in turn shows that $\ell^2 = \ell^1 + \calO(x^4)$. Thus for computing third derivatives, we may as well work 
with $\ell^2$. The second equality yields
\[
\ell^2 = \bar{\ell} \, \frac{\sin \omega}{\sin(\pi-\beta - \omega)} = \frac{\bar{\ell}}{\cos\beta}\left(
1 - \tan \beta \cot \omega + \calO(x^4)\right). 
\]
Recalling that $\bar{\ell} = u_0 + \psi_t$, and using the expressions for $\beta$ and $\omega$ above, we obtain
\[
\begin{array}{rcl}
\ell^2 & = & \left(\psi_t + u_0\right) \left(1 + \calO(x^4)\right)\left(
1 - (\frac12 \del_s \kappa_0 x^2 + \ldots)(\del_s \phi_t + \calO(t^2)\right) \\
& = & \psi_t + \frac12 \kappa_0 x^2 + u_{3,0} x^3 + \calO(t^2) + \calO(x^4),
\end{array}
\]
which yields (\ref{eq:derivs}).

We turn to the computation of $\tilde{\chi}$. The idea is much the same.  Let $(s_0,x_0)$ denote
the cylindrical coordinates of $P$, i.e.\ as a horizontal graph over the cylinder $\Gamma_0$, and
$(s_0',x_0')$ the cylindrical coordinates of the point $\tilde{R} = \tilde{\lambda} \cap Y_t$, when 
written as a horizontal graph over $\Gamma_t$. Since $\tilde{R} = P + \phi_t(s,x) \olnu$ and 
$\olnu = \olN_0 - \kappa_0 x \del_x + \calO(x^2)$, we obtain
\begin{equation}
s_0' = s_0 + \calO(x^2), \qquad x_0' = x_0 - \kappa_0 x \phi_t + \calO(x^2).
\label{eq:newcoords}
\end{equation}
We shall compute the horizontal displacement $u_t(s_0',x_0')$ of $Y_t$ from $\Gamma_t$
in the plane $x = x_0'$. For this, let $\tilde{\gamma}_t = \Gamma_t \cap \{x = x_0'\}$, which is just
the vertical translate of $\gamma_t$ into this plane. Abusing notation slightly, write $\tilde{P}$ for
the point on $\tilde{\gamma}_t$ which is directly below the intersection of $\tilde{\lambda} \cap \Gamma_t$ 
(which had been called $\tilde{P}$ above). Finally, let $\tilde{Q}$ denote the point on $\tilde{\gamma}_t$ 
which is nearest to $\tilde{R}$, i.e.\ so that $\tilde{R} = \tilde{Q} + c \olN_t$ for some constant $c$.

Consider the thin triangle $\tilde{P}\tilde{Q}\tilde{R}$. The sides $\tilde{P}\tilde{R}$ and $\tilde{R}\tilde{Q}$ have 
lengths $\tilde{\chi} \cos\alpha $ and $u_t(s_0',x_0')$, respectively. Let $\eta$ and $\zeta$ be the angles at the vertices 
$\tilde{P}$ and $\tilde{Q}$. We find that
\[
\eta = \frac{\pi}{2} + \calO(t) + \calO(x^2), \qquad \zeta = \frac{\pi}{2} + \calO(x^2).
\]
By the law of sines again, $u_t(s_0',x_0')/\sin \eta = \tilde{\chi}/\sin \zeta$, so 
\[
\tilde{\chi} \cos\alpha  = u_t(s_0',x_0') (1 + \calO(x^4))(1 + \calO((t + x^2)^2)) = u_t(s_0',x_0') + \calO(t^2) + \calO(x^4).
\]
Expanding in a Taylor series, using (\ref{eq:newcoords}) and recalling that $u_t = \calO(x^2)$, we find that
\[
\begin{array}{rcl}
u_t(s_0',x_0') & = &  u_t(s_0,x_0) + \del_s u_t (s_0,x_0) \calO(x^2) + \del_x u_t(s_0,x_0) (-\kappa_0 x \phi_t + \calO(x^2)) +
\calO(t^2 + x^4)  \\
& = & u_t(s_0,x_0) - \kappa_0 \kappa_t \phi_t x^2 - 3 \kappa_0 u_{3,t} \phi_t x^3 + \calO(t^2 + x^4) \\
& = & (\frac12 \kappa_t - \kappa_0 \kappa_t \phi_t) x^2 + (u_{3,t} - 3 \kappa_0 u_{3,t} \phi_t )x^3 + \ldots;
\end{array}
\]
also, as before, $1/\cos \alpha = 1 + \calO(x^2)$, so altogether we obtain
\[
\left. \del_t \del_x^3 \tilde{\chi} \right|_{t=x=0} = 6 \dot{u}_{3,0} - 18 \kappa_0 u_{3,0} \dot{\psi}
\]

We have now proved that 
\[
\left. \del_t \del_x^3 \phi_t \right|_{x=t=0} = \left. \del_t \del_x^3 (\tilde{\ell} + \tilde{\chi}) \right|_{x=t=0} = 
6 \calD_{Y_0}(\dot{\phi}_0) = 6 \dot{u}_{3,0}+18 \kappa_0 u_{3,0} \dot{\psi}.
\]
Inserting this into (\ref{eq:2ndderiv}) yields the formul\ae\ in the statement of the theorem.
\end{proof}

\section{Critical points of $\calA$ in $\HH^3$}
In this section we prove that the only nondegenerate critical points of renormalized area for proper
minimal surfaces in all of $\HH^3$ are the totally geodesic planes, the boundary curves of which are
circles. The proof requires a preliminary geometric lemma about osculating circles of plane curves,
which is perhaps of independent interest, and then proceeds via a refined version of the asymptotic
maximum principle.

Recall that the osculating circle $C$ at a point $p \in \gamma$ is a circle which
makes second order contact with $\gamma$ at that point; its curvature, the inverse
of its radius, is therefore the same as that for $\gamma$ at this point of intersection.
By inscribed we simply mean that $C$ remains entirely within the closure of one of the
two components of $S^2 \setminus \gamma$.

\begin{proposition}
\label{oscullating} Let $\gamma$ be a $\calC^2$ embedded loop in $S^2= \del \HH^3$
and $\Omega$ one component of $S^2 \setminus \gamma$. Then there exists a point
$p \in \gamma$ so that the osculating curve of $\gamma$ at $p$ lies in
$\overline{\Omega}$.
\end{proposition}
\begin{remark}
As the proof makes clear, there may be many such points. Of
course, we could equally well have replaced $\Omega$ by the other
component of $S^2 \setminus \gamma$.
\end{remark}
\begin{proof}
We begin with a few elementary observations. First, the entire question is invariant under
M\"obius transformations, hence we may freely apply such transformations to reduce the problem
to one that is easier to visualize. Second, if $C'$ is any circle inscribed in $\Omega$ which
is locally maximal (in the sense that for any continuous family of inscribed circles $C'(\e)$
with $C'(0) = C'$, the family of radii $r'(\e)$ reaches a local maximum at $\e=0$) then necessarily
either $C'$ is tangent to $\gamma$ at two or more distinct points, or else $C'$ is tangent to
$\gamma$ at a single point and is the osculating circle there.  The reason is that if there is
only one point of contact, $p = C' \cap \gamma$, then the curvature of the circle $1/r'$ is
greater than or equal to $\kappa(p)$, the curvature of $\gamma$ at $p$. If this inequality is
strict and there are no other points of contact, 
then we could increase the radius of $C'$ slightly while keeping it inside $\Omega$.

Thirdly, and slightly more complicated, we claim that if $C'$ arises as a limit of inscribed
circles $C'_j$ such that each $C'_j \cap \gamma$ contains at least two points $P_j \neq Q_j$
and $\mbox{dist}\,(P_j ,Q_j) \to 0$ as $j \to \infty$, then the limit $C'$ is necessarily an
inscribed osculating circle. To see this, choose for each $j$ a M\"obius transformation $F_j$ which
carries $C_j'$ into a fixed straight line in $\RR^2$, say the $y_1$-axis. We also suppose that
$\mbox{dist}\,(F_j(P_j),F_j(Q_j))= \mbox{dist}\,(P_j,Q_j)$, so that $F_j$ does not diverge,
that $F_j$ carries $\Omega$ to the lower half-plane, and that $\lim F_j(P_j) = \lim F_j(Q_j)$ is
the origin. Each curve $F_j(\gamma)$ lies in the upper half-plane and is tangent to the $y_1$-axis
at two points which are converging to the origin. Clearly, the curve must be a graph over the axis
between these two points for $j$ large enough, say of some function $f_j$. By the intermediate value
theorem, there is a sequence of points $t_j \to 0$ such that $f_j''(t_j) = 0$. Taking a limit, we
see that the limiting curve is flat to second order at the origin and lies entirely in the closed
upper half-plane, which proves the claim.

We can now proceed with the proof. Let $\calK$ denote the set of pairs $(P,Q) \in \gamma \times \gamma$,
$P \neq Q$ such that there is an inscribed circle which is tangent to $\gamma$ at these two points
(and possibly other points as well).  By the second remark above, if this set were empty, then there
would have to be an inscribed osculating circle already. So assume $\calK \neq \emptyset$. We claim
that the closure of $\calK$ must intersect the diagonal, which by the third remark above would produce
an inscribed osculating circle: If it did not, then $\calK$ would be compact in $\gamma \times \gamma$,
and hence there would be a point $(P',Q') = (\gamma(t_1),\gamma(t_2))$ such that the difference in parameter
values $|t_2-t_1|$ is minimal (for some fixed parametrization of the curve). Let $C'$ be the corresponding
circle. Conformally transform so that $C'$ is the $y_1$-axis. Then the images of $P'$ and $Q'$ lie on this
axis and the transformed curve lies entirely on or above the axis. If the image of $P'$ is the left-most
point of tangency of the curve with the $y_1$-axis, fix another point $P''$ just to the right for which
the horizontal component of the downward pointing normal is positive. There is a maximal radius for which
a circle tangent to the curve at $P''$ remains in the component of the lower half-plane, and this circle
is obviously tangent to the curve at another point $Q''$ to the left of the image of $Q'$. This shows the
existence of another pair $(P'',Q'')$ strictly between the pair $(P',Q')$ for which there is an inscribed
circle tangent at these two points, which is a contradiction. This finishes the proof.
\end{proof}

Now we turn to the main result of this section.

\begin{theorem}
Let $Y \in \calM_k(\HH^3)$ be nondegenerate and suppose that $\del Y = \gamma$
is $\calC^{3,\alpha}$ and connected, and $Y$ is a critical point for $\calA$. Then $\del Y$ is a
round circle and $Y$ is a totally geodesic disk.
\end{theorem}
\begin{remark}
Using Corollary \ref{co:crit}, the proof below actually shows that if $Y$
is critical for the extended renormalized area functional $\calR$, then
the same conclusion holds.
\end{remark}

\begin{proof}
First, by the results in the previous section, we know that $Y$
must be a minimal surface and also that the formally undetermined
term $u_3$ in its expansion must vanish.

\par Using the result about osculating circles, and applying a
conformal transformation, we reduce to the case where $\gamma$ is
a closed curve in $\RR^2$ which lies entirely in the closed upper
half-plane, and which is tangent to second order to the $y_1$-axis
at the origin.  We now write some neighbourhood of $0 \in Y$ as a
horizontal graph over the $(y_1,x)$-plane, i.e.\ $Y= \{y_2 =
u(y_1,x)\}$ for $|y_1| < \delta$, $x < \delta$. Set $s = y_1$ for
simplicity. This function satisfies the minimal surface equation,
which in this coordinate system takes the form
\begin{equation}
\label{pde} \calF(u) = (1 + u_x^2) u_{ss} -  2u_s u_x u_{sx} + (1
+ u_s^2)u_{xx} - \frac{2(1+u_x^2+u_s^2)}{x} u_x =0.
\end{equation}
(Note that this is simply the specialization of (\ref{eq:mse}) when the base of the cylinder
$\Gamma$ is flat, so $\kappa = 0$ and $w = 1$.) Furthermore, since $Y$ lies to one side of the
$(s,x)$ plane, $u \geq 0$ everywhere, and it has the asymptotic expansion
\begin{equation}
u(s,x) \sim a(s) x^2 + \calO(x^{3+\alpha}).
\label{aeu}
\end{equation}
The absence of the $x^3$ term is because $Y$ is critical for $\calA$; furthermore, $a(0)$ is the one half the
curvature of $\gamma$ at $0$, and hence because it osculates the line there, $a(0) = 0$.

Now choose any $\beta \in (0,\alpha)$ and define $u_c = u - c x^{3+\beta}$ for $c > 0$. By the various
properties above, if we fix $\delta$ then choose $c$  sufficiently small, the function
$u_c \geq 0$ on all four sides of the rectangle $|s| \leq \delta$, $0 \leq x \leq \delta$.
However, using (\ref{aeu}) and the fact that $\beta < \alpha$, we also have
that $u_c(0,x) < 0$ for $x$ sufficiently small.  This means that the minimum
of $u_c$ is achieved somewhere strictly inside this rectangle, and of course at that point
$(s_c,x_c)$, we have $\calF(u_c)(s_c,x_c) \geq 0$.

On the other hand, we compute that
\[
\calF(u - c x^{3+\beta}) = \calF(u) - L_u (c x^{3+\beta}) + \calO(c x^{6 + 2\beta});
\]
the first term on the right vanishes, while the operator in the second term is the linearization
of $\calF$ at $u$,
\begin{eqnarray*}
L_u v & = & (1 + u_x^2) v_{ss} - 2 u_s u_x v_{sx} + (1+u_s^2)v_{xx} + 2(u_x u_{ss} - u_s u_{sx})v_x + v_s \\
& & + 2(u_s u_{xx} - u_x u_{sx})v_s - \frac{2(1+u_x^2+u_s^2)}{x} v_x -\frac{4u_x^2}{x}v_x-\frac{4u_su_x}{x}v_s.
\end{eqnarray*}
By (\ref{aeu}) and scale-invariant Schauder estimates, $x^{-2}|u| + x^{-1}|\nabla u| + |\nabla^2 u| \leq K$.
Hence altogether
\[
\calF(u_c) \leq - c \beta (3 + \beta) x^{1+\beta} + Kc x^{2+\beta},
\]
where $K$ is independent of $c$ and $\delta$. Choosing $\delta$ sufficiently small, we can ensure that
$\calF(u_c) < 0$ everywhere in this box, which contradicts that it is positive at $(s_c,x_c)$.
This proves the theorem.
\end{proof}

\section{Connections with the Willmore functional}
It is not particularly surprising that the functional $\calA$ is connected with the Willmore functional,
which by definition is the total integral of the square of mean curvature. In this final section we
explore some of these relationships.

Fixing a special bdf $x$, then the expansion  of $\olg = x^2 g$ has only even powers, so
its natural extension to a $\ZZ_2$-invariant metric on the double of $M$ across its boundary
is smooth; furthermore, any surface $Y \in \widetilde{M}_k(M)$ can also be doubled to
a closed $\calC^{2,1}$ surface. We denote these doubles by $2M$ and $2Y$, respectively.

\begin{proposition}
If $Y \in \calM_k(M)$, then
\[
\calA(Y) = - \frac12 \calW(2Y).
\]
\end{proposition}
\begin{proof}
We have already noted that $|\wh{k}|^2\, dA_g = |\wh{\overline{k}}|^2\, dA_{\olg}$. Next, observe that
\[
\left|\wh{\overline{k}}\right|^2 = 2 \left|\frac{\overline{\kappa}_1 - \overline{\kappa}_2}{2}\right|^2 =
2\left(\overline{H}^{\,2} - \overline{\kappa}_1\overline{\kappa}_2\right),
\]
where the $\overline{\kappa}_j$ are the principal curvatures of $Y$ with respect to $\olg$. Thus finally
\begin{eqnarray*}
\calA(Y) &= & -2\pi \chi(Y) - \int_Y \left(\overline{H}^{\,2} -
\overline{\kappa}_1\overline{\kappa}_2\right)\, dA_{\olg} \\ & = &-2\pi \chi(Y) -
\frac12 \int_{2Y} \left(\overline{H}^{\,2} - \overline{\kappa}_1\overline{\kappa}_2\right)\, dA_{\olg}
= -2\pi \chi(Y) - \frac12 \calW(2Y) + 2\pi \chi(Y),
\end{eqnarray*}
since $\chi(2Y) = 2\chi(Y)$.
\end{proof}

On the other hand, it does not seem to be the case that the extended functional $\calR$ on
$\widetilde{\calM}_k(M)$ has any simple connection with the Willmore functional. Nonetheless, this
result recasts the program of finding extremals of $\calA$ in a different and more classical light.
Let us conclude by explaining this a bit further.

The conformal invariance of the Willmore functional makes its variational theory quite
challenging. The existence of minimizers of the Willmore functional for closed surfaces
in $\RR^3$ with genus less than or equal to some fixed constant is proved in \cite{Si}, and
this result is sharpened in \cite{Ku}, where it is shown how to prevent a `drop of genus' in these
minimization arguments. We have just shown that the problem of finding extrema, and in
particular, maxima, for $\calA$ on some given moduli space $\calM_k(M)$ is equivalent
to finding constrained extrema (in particular, minima) of $\calW$ on $2M$ with respect
to the metric $\olg$, within the {\it restricted} class of surfaces which are invariant under
the $\ZZ_2$ involution and which are minimal with respect to $g$.  If it were possible
to adapt the arguments from \cite{Si} to this ambiently curved setting, we could prove the existence
of such extrema. This may well be subtle, and the strengthened result from \cite{Ku}
may not be available, even when $M = \HH^3$. Indeed, consider that we have proved
that there are no nondegenerate critical points for $\calA$ on $\calM_k(\HH^3)$
when $k > 0$. This indicates that any extremizing sequence $Y_j$, e.g.\ one for
which $\calA(Y_j)$ tends to the supremum, probably does not converge to a surface of the
same genus.

\end{document}